%% file: article.tex
\documentclass[twoside, a4paper,10pt]{article}

\include{include}

\begin{document}

\title{Asphericity and small cancellation theory for rotation families of groups.}
\author{R\'emi Coulon \\ \small{IRMA, Universit\'e de Strasbourg}\\\small{7 rue Ren\'e Descartes, 67084 Strasbourg Cedex, France} \\ \small{\texttt{coulon@math.unistra.fr}}}
\date{}

\maketitle

\abstract{
	Using small cancellation for rotating families of groups, we construct new examples of aspherical polyhedra. This paper has been accepted in ``Groups, Geometry, and Dynamics''.
}

\tableofcontents

\input{0-introduction}
\input{1-hyperbolic_spaces}
\input{2-cone}
\input{3-cone-off}

\input{4-small_cancellation}

\input{5-examples}

\bibliography{bibliography}
\bibliographystyle{abbrv}

\newpage
\listoftodos[List of corrections]

\end{document}

%% file: include.tex
\usepackage[utf8]{inputenc}
\usepackage[T1]{fontenc}
\usepackage[dvips]{graphicx} 
\usepackage{amsmath, amsthm, amsfonts, amssymb}
\usepackage{latexsym}
\usepackage{color}
\usepackage[all]{xy}
\usepackage{stmaryrd}
\usepackage[dvips]{epsfig}
\usepackage{endnotes}
\usepackage[colorinlistoftodos, disable]{todonotes}
\usepackage{setspace}
\usepackage{color}
\usepackage{tikz}

\newtheorem{defi}{Definition}[subsection]
\newtheorem{lemm}[defi]{Lemma}
\newtheorem{prop}[defi]{Proposition}
\newtheorem{theo}[defi]{Theorem}
\newtheorem{coro}[defi]{Corollary}
\newtheorem{pref}[defi]{Proposition-definition}

\newtheorem*{theovide}{Theorem}

\newtheorem{defisec}{Definition}[section]
\newtheorem{lemmsec}[defisec]{Lemma}
\newtheorem{propsec}[defisec]{Proposition}
\newtheorem{theosec}[defisec]{Theorem}

\makeatletter
\newcommand{\paragraphit}{
	\@startsection{paragraphit}{4}{0mm}{\baselineskip}{\baselineskip}{\normalfont\normalsize\itshape}
}
\makeatother

\newcommand{\rem}{\paragraph{Remark :}}
\newcommand{\rems}{\paragraph{Remarks :}}
\newcommand{\ex}{\paragraph{Example :}}
\newcommand{\exs}{\paragraph{Examples :}}
\newcommand{\nota}{\paragraph{Notation :}}
\newcommand{\voc}{\paragraph{Vocabulary :}}

\newcommand{\N}{\mathbf{N}}

\newcommand{\R}{\mathbf{R}}
\newcommand{\C}{\mathbf{C}}

\newcommand{\HR}[1]{\mathbf{H}_{#1}(\R)}
\newcommand{\HC}[1]{\mathbf{H}_{#1}(\C)}

\newcommand{\HP}[1]{\mathbf{H}_{#1}}

\newcommand{\resp}{respectively\ }

\newcommand{\diam}{\operatorname{diam}}
\newcommand{\arccosh}{\operatorname{arccosh}}

\newcommand{\cyl}{\operatorname{cyl}}
\newcommand{\id}{\operatorname{id}}
\renewcommand{\epsilon}{\varepsilon}
\renewcommand{\phi}{\varphi}

\newcommand{\oasly}{$\omega$-almost surely ($\omega$-as)\renewcommand{\oasly}{$\omega$-as}}
\newcommand{\oeb}{$\omega$-essentially bounded ($\omega$-eb)\renewcommand{\oeb}{$\omega$-eb}}

\newcommand{\intval}[2]{\left\{ {#1}, \dots, {#2} \right\}}

\newcommand{\gro}[3]{\left< #1,#2\right>_{#3}}
\newcommand{\limo}{\lim_\omega}
\newcommand{\dist}[3]{\left| #3- #2 \right|_{#1}}
\newcommand{\distV}[1]{\left| \ . \ \right|_{#1}}
\newcommand{\distX}[2]{\left| #2- #1 \right|}
\newcommand{\distSC}[2]{\left\| #2- #1 \right\|}

\newcommand{\distL}[2]{d\left(#1,#2\right)}
\newcommand{\distLbis}[3]{d_{#1}\left(#2,#3\right)}
\renewcommand{\angle}[2]{\theta\left( #1 , #2\right)}
\newcommand{\geo}[2]{\left[ #1, #2 \right]}
\newcommand{\rinj}[2]{r_{\textit{inj}}\left(#1,#2\right)}
\newcommand{\stab}[1]{\operatorname{Stab}\left( #1\right)}
\newcommand{\rips}[1]{P_d\left( #1\right)}
\newcommand{\ripsn}[2]{P_d^{(#1)}\left( #2\right)}

\def\restriction#1#2{\mathchoice
              {\setbox1\hbox{${\displaystyle #1}_{\scriptstyle #2}$}
              \restrictionaux{#1}{#2}}
              {\setbox1\hbox{${\textstyle #1}_{\scriptstyle #2}$}
              \restrictionaux{#1}{#2}}
              {\setbox1\hbox{${\scriptstyle #1}_{\scriptscriptstyle #2}$}
              \restrictionaux{#1}{#2}}
              {\setbox1\hbox{${\scriptscriptstyle #1}_{\scriptscriptstyle #2}$}
              \restrictionaux{#1}{#2}}}
\def\restrictionaux#1#2{{#1\,\smash{\vrule height .8\ht1 depth .85\dp1}}_{\,#2}}

\renewcommand{\leq}{\leqslant}
\renewcommand{\geq}{\geqslant}

\definecolor{monbleu}{rgb}{0.25, 0.25, 0.75}

\newcommand{\cmath}[1]{\todo[noline, color=blue!20]{#1}}

%% file: 0-introduction.tex
\section*{Introduction}

\paragraph{} 
The goal of this article is to produce new examples of aspherical polyhedra.
The construction we have in mind is the following: let $P$ be an aspherical simplicial complex and $Q$ a subcomplex of $P$.
We define $\bar P$ by attaching to $P$ a cone of base $Q$.
We are looking for conditions under which $\bar P$ remains aspherical.
This  kind of situation was already studied by J. H. C. Whitehead in the 1940's.
In \cite{Whi41} and \cite{Whi46}, he studied the second homotopy group of a space $\bar P$, obtained by attaching 2-cells to a cell complex $P$.
He proved that $\pi_2\left(\bar P\right)$ exactly describes the identities between the relators that define the projection $\pi_1\left( P\right) \rightarrow \pi_1\left( \bar P\right)$.
In this paper we try to attach higher dimensional cells.

\paragraph{} 
Our main example the following.
Let $\HC n$ be the complex, $n$-dimen\--sional, hyperbolic space.
We consider $SO(n,1)$ as the stabilizer of the real hyperbolic space $\HR n$ in $\HC n$.
Let $G \subset SU(n,1)$ be a  real lattice, i.e. a lattice of $SU(n,1)$ such that $H=G\cap SO(n,1)$ is still a lattice of $SO(n,1)$.
We want to study the group $G / \ll H\gg$.

\begin{theovide}
There exists a finite index subgroup $G'$ of $G$ with the following property.
Let $H'$ be the group $G' \cap SO(n,1)$.
Let $\bar P$ be the space obtained by attaching to $\HC n/G'$ a cone of base $\HR n /H'$.
The complex $\bar P$ is a classifying space for the group $\bar G' = G'/ \ll H' \gg$.
\end{theovide}

This situation is similar to a result of N. Bergeron  in \cite{Ber03} who proved that the map from the homology associated with $\HR n / G \cap SO(n,1)$ to the one of $\HC n /G$ is one to one.
We also establish some analogue statements to our theorem for the following pairs $\left( SO(n,1), SO(k,1)\right)$, $\left( SU(n,1), SU(k,1)\right)$ and $\left( Sp(n,1), Sp(k,1)\right)$.

\paragraph{} Our strategy is the following. We endow the space $\bar P$ with a local hyperbolic geometry and apply a version of the Cartan-Hadamard theorem.
This will prove that the universal cover $\bar X$ of $\bar P$ is globally hyperbolic. 
We deduce then the asphericity from a kind of Rips' theorem (cf. Th. \ref{aspherical hyperbolic complex}):
if $\bar X$ is a hyperbolic simplicial complex, it is sufficient to prove that $\bar X$ is locally contractible.
This last local assumption follows from the property of the developing map.

\paragraph{} The hyperbolic structure on $\bar P$ is constructed as follows. 
Given a subcomplex $Q$ of $P$, we endow the cone of base $Q$ with a metric modelled on the hyperbolic disc. 
An argument of Berestovskii tells us that if $Q$ is CAT(1) then the cone is CAT(-1) (see \cite{BriHae99}).
In particular it is hyperbolic.
The problem is then to find some conditions so that the complex $\bar P$ remains locally hyperbolic.

\paragraph{} With this aim in mind, we explore an idea of M. Gromov (see \cite{Gro03}) to extend the small cancellation theory to a so called  ``rotation family'' of groups. 
If $X$ is a metric space and $G$ a group acting on $X$ by isometries, a rotation family is a pairwise distinct collection $\left(Y_i, H_i\right)_{i \in I}$ such that
\begin{itemize}
	\item $H_i$ is a subgroup of $G$ stabilizing $Y_i \subset X$, 
	\item there is an action of $G$ on $I$ compatible with the one on $X$ (i.e. for all $g \in G$, for all $i \in I$, $Y_{gi} = gY_i$ and $H_{gi} = gH_i g^{-1}$).
\end{itemize}
In order to study such a family, we define two quantities that respectively play the role of the largest piece and the smallest relator in the usual small cancellation theory.
The constant $\Delta$ measures the overlap between two $Y_i$'s whereas $\rho$ is the minimal translation length of a non-trivial element that belongs to an $H_i$.

\paragraph{}
Given a rotation family, we define the cone-off over $X$, denoted by $\dot X$, by attaching to $X$ cones of base $Y_i$.
We consider the space $\bar P = \dot X / G$ whose fundamental group is $\bar G = G /\ll H_i \gg$.
The small cancellation provides us a framework where it is possible to endow $\bar P$ with a local hyperbolic geometry.
Moreover this theory recovers the usual small cancellation (see \cite{LynSch77} or \cite{Olc91a}) and the small cancellation with graphs as well (see \cite{Gro03} or \cite{Oll06}).

\paragraph{} We describe the space $\bar P$  as an orbifold using two kind of charts: the cones and the cone-off. 
Section \ref{part cone} is dedicated to the study of the geometry of the cones.
Adapting an argument of Berestovskii, we prove that the cone over a hyperbolic space remains hyperbolic (cf. Th. \ref{first hyperbolicity theorem}).
Section \ref{part cone-off} deals with the cone-off $\dot X$.
In particular we prove the following fact. 
Under small cancellation assumptions, the cone-off over a hyperbolic space is still hyperbolic (cf. Th. \ref{second hyperbolicity theorem}).
We choose for this proof an asymptotical point of view that involves ultra-limits as in \cite{Dru02}.
The goal of the main technical lemma is to switch the cone-off construction and the ultra-limit:
given a sequence of metric spaces $X_n$, we prove (see Cor. \ref{local isometry cone off}) that there exists a local isometry between the cone-off over the ultra-limit of $X_n$ and the ultra-limit of $\dot X_n$.
Section \ref{part small cancellation} mixes all the previous ingredients in order to obtain the following theorems.
\begin{theovide}[cf Th. {\ref{very small cancellation}}]
There exist two positive numbers $\delta_0$ and $\Delta_0$ satisfying the following property.
		
		Let  $X$ be a geodesic, simply connected, $\delta$-hyperbolic space and  $G$ a group acting properly, by isometries on $X$. Let $(Y_i,H_i)_{i \in I}$ be a rotation family, such that each $Y_i$ is strongly-quasi-convex.

		Let $N$ be the normal subgroup of $G$ generated by the $H_i$'s and $\bar{G}$ the quotient group $G/N$. Assume also that 
		\begin{displaymath}
		\frac  \delta \rho \leq \delta_0 \ \text{ and } \ \frac{\Delta}\rho\leq \Delta_0
		\end{displaymath}
		
		Then the universal cover of $\bar P$, $\bar X$, is hyperbolic and $\bar{G}$ acts properly by isometries on $\bar X$.
		
		Moreover if $G$ (\resp $H_i$) acts co-compactly on $X$ (\resp $Y_i$) and $I/G$ is finite, then $\bar X / \bar G$ is compact. In particular $\bar G$ is hyperbolic.
\end{theovide}

\begin{theovide}[cf Th. {\ref{asphericity universal cover}}]
	Under the same hypotheses, if $X$ is a $n$-dimen\-sional simplicial complex such that every closed ball of $X$ and of each $Y_i$ is contractible in its appropriate neighbourhood, then $\bar X$ is contractible.
\end{theovide}

\rem In some cases, the space $\bar P = \bar X / \bar G$ may be endowed with a sharper geometry than the hyperbolic one. 
For instance, M. Gromov constructed in a similar way CAT(-1) polyhedra in order to produce infinite torsion groups (see \cite[Chap 12]{GhyHar90}).
In \cite{Gro01} M. Gromov introduced the notion of {CAT(-1,$\epsilon$)} spaces, ae $\epsilon$-perturbated version of CAT(-1)-spaces.
This provides another framework to study small cancellation constructions that is not asymptotic.

\paragraph{} In Section \ref{part examples}, we explain how to construct examples of rotation families that satisfy the small cancellation assumptions. 
To that end, we use the geometry of lattices and a result of N. Bergeron about the profinite topology of finitely generated linear groups \cite{Ber00}.
This leads to these new examples of aspherical polyhedra.

\paragraph{Question :}
The small cancellation for rotating families provides a framework to study quotient groups which looks very similar to the usual small cancellation.
However the groups obtained in this way may have very different properties. 
For instance, a usual $C'\left( \frac 16 \right)$ small cancellation group acts properly discontinuously and co-compactly on a CAT(0)-cubical complex. (see \cite[Th. 1.2]{Wis04}).
In particular it cannot have Kazhdan's property (T) unless it is finite and cyclic.
On the other hand, M. Gromov used in \cite{Gro03} (see also \cite{ArzDel00}) the small cancellation theory with graphs in order to construct Kazhdan groups.
To that end he embedded expanding graphs in the Cayley graph of the group.
We wonder if it is possible to construct other examples of Kazhdan groups using the small cancellation theory with rotating families.
In particular we are interested in the following example.
Let $G$ be, as previously, a real lattice of $SU(n,1)$ and $H$ the subgroup $G \cap SO(n,1)$. 
Does $\bar G = G / \ll H \gg$ have the Kazhdan property (T) ?

\paragraph{Acknowledgements.} I am grateful to Thomas Delzant for his invaluable help and advice during this work. Many thanks go to Nicolas Bergeron who explained me how to use the profinite topology.  I thank also the referee for many useful comments and corrections.

%% file: 1-hyperbolic_spaces.tex
\section{Hyperbolic spaces}
\label{part hyperbolic spaces}

Let $X$ be a metric space.
If $x$ and $x'$ are two points of $X$, we denote by $\dist X x {x'}$ (or simply $\distX x {x'}$) the distance between them. 
Although it may not be  unique, we denote by $\geo x {x'}$ a geodesic joining $x$ and $x'$.
Given a base point $x$, the Gromov product of two points $y$ and $z$ is defined by 
\begin{displaymath}
	\gro y z x = \frac 12 \Big( \distX x y + \distX x z - \distX y z \Big)
\end{displaymath}
Let $\delta$ be a non negative  number. The space $X$ is $\delta$-hyperbolic if for all $x,y,z,t \in X$, we have
\begin{math}
	\gro x z t \geq \min \left\{ \gro x y t , \gro y z t \right\} - \delta
\end{math}.
$\R$-trees are very special examples of hyperbolic spaces. 

\begin{pref}[{\cite[Chap 2. Prop 6]{GhyHar90}} or  {\cite[Chap. 3 Th. 4.1]{CooDelPap90}}]

	An $\R$-tree is a geodesic space such that  every two points are connected by a unique topological arc. A metric space is an $\R$-tree if and only if it is geodesic and 0-hyperbolic.
\end{pref}

\begin{defi}[Quasi-isometry]
	Let $\eta$ be a non negative number. A $(1,\eta)$-quasi-isometry is a map $f : X \rightarrow Y$ between two metric spaces such that for all $x, x' \in X$, we have 
\begin{displaymath}
	\distX {x}{x'} - \eta \leq \distX {f(x)}{f(x')} \leq \distX {x}{x'} + \eta
\end{displaymath}
\end{defi}

The next result is a very easy case of the stability of quasi-geodesics.
An asymptotic proof of this fact for a general $(\lambda, k)$-quasi-isometry can be found in \cite{Cou}.

\begin{prop}
\label{quasi-isometry hyperbolicity}
	Let $\delta$ be a non negative number. 
	For all $\delta' > \delta$, there exists $\eta >0$ satisfying the following property. 
	Let $X$  be  a metric space and $Y$ a $\delta$-hyperbolic space. 
	If there exists a $(1,\eta)$-quasi-isometry from $X$ to $Y$, then $X$ is $\delta'$-hyperbolic.
\end{prop}

\begin{proof}
	Let $f : X \rightarrow Y$ be a $(1,\eta)$-quasi-isometry.
	For all $x,y,z \in X$ we have 
	\begin{displaymath}
		\gro {f(x)}{f(y)}{f(z)} -\frac 3 2 \eta \leq \gro x y z \leq \gro {f(x)}{f(y)}{f(z)} +\frac 3 2 \eta
	\end{displaymath}
	It follows that for all $x,y,z,t \in X$,
	\begin{eqnarray*}
		\gro x z t & \geq &\gro {f(x)}{f(z)}{f(t)} - \frac 32 \eta \\
						& \geq &\min \left\{  \gro {f(x)}{f(y)}{f(t)},  \gro {f(y)}{f(z)}{f(t)} \right\} - \delta - \frac 32 \eta \\
						& \geq &\min \left\{ \gro x y t, \gro y z t\right\} - (\delta  + 3 \eta) \\
	\end{eqnarray*}
	Hence $X$ is $(\delta + 3 \eta)$-hyperbolic.
\end{proof}

\subsection{Ultra-limits of hyperbolic spaces}

\paragraph{}Let us recall the definition of the ultra-limit of a sequence of pointed metric spaces and its link with hyperbolicity. For more details about this point of view see \cite{Dru01}, \cite{Dru02} or \cite{DruSap05}.

\paragraph{}A non-principal ultra-filter is a finite additive map $\omega : \mathcal P (\N) \rightarrow \left\{0, 1\right\}$ which vanishes on every finite subset of $\N$ and such that $\omega(\N) =1$. 
A property $\mathcal P_n$ is true \oasly\ if 
\begin{math}
	\omega \big( \left\{ n \in \N / \mathcal P_n \text{ is true }\right\} \big) = 1
\end{math}.
A real sequence $\left( u_n\right)$ is \oeb\ if there exists $M \in \R$, such that $\left|u_n\right| \leq M$, \oasly.
If $l$ is a real number, we say that the $\omega$-limit of $\left( u_n\right)$ is $l$ and write $\limo u_n = l$, if for all $\epsilon >0$, $\distX {u_n} l \leq \epsilon$, \oasly. 
In particular, any \oeb\ sequence admits an $\omega$-limit (cf. \cite{Bou71}).

\paragraph{} Let $\left( X_n, x^0_n \right)$ be a sequence of pointed metric spaces. We consider
\begin{displaymath}
	\Pi_\omega X_n = \left\{ (x_n) / \forall n\in \N,\ x_n\in X_n \text{ and } \left(\distX {x^0_n}{x_n}\right) \text{ is \oeb.}\right\}
\end{displaymath}

We endow this space with a pseudo-metric defined as follows:
\begin{displaymath}
	\distX {(x_n)}{(y_n)} = \limo \distX {x_n}{y_n}
\end{displaymath}

\begin{defi}[Ultra-limit of metric spaces]
	Let $\left(X_n, x^0_n\right)$ be  a sequence of pointed metric spaces and  $\omega$ a non-principal ultra-filter. The $\omega$-limit of $\left(X_n, x^0_n\right)$, denoted by $\limo \left(X_n, x^0_n\right)$ (or simply $\limo X_n$) is the quotient of $\Pi_\omega X_n$ by the equivalence relation which identifies the points at distance zero.
\end{defi}

The pseudo-distance on $\Pi_\omega X_n$ induces a distance on $\limo X_n$. 
\nota 

\begin{enumerate}
	\item If $(x_n)$ is an element of $\Pi_\omega X_n$, its image in $\limo X_n$ is denoted by $\limo x_n$.
	\item For all $n \in \N$, let $Y_n$ be a subset of $X_n$. The set $\limo Y_n$ is defined by 
	\begin{displaymath}
		\limo Y_n = \left\{ \limo y_n / \left(\distX {x_n^0} {y_n}\right) \text{ is \oeb\ and } y_n \in Y_n \text{ \oasly}\right\}
	\end{displaymath}
\end{enumerate}

\begin{prop}
	\label{ultra limit of hyperbolic spaces}
	Let $\omega$ be a non-principal ultra-filter. 
	Let $(\delta_n)$ be a sequence of non negative numbers which admits a $\omega$-limit $\delta$. 
	Let $\left(X_n, x^0_n\right)$  be a sequence of pointed metric spaces. 
	If for all $n \in \N$, $X_n$ is $\delta_n$-hyperbolic, then  the limit space $\limo X_n$ is $\delta$-hyperbolic.
\end{prop}

\begin{proof}
	Let $x=\limo x_n$, $y= \limo y_n$, $z= \limo z_n$ and $t= \limo t_n$ be four points of $\limo X_n$. 
Since $X_n$ is $\delta_n$-hyperbolic, we have for all $n \in \N$,
\begin{math}
	\gro {x_n}{z_n}{t_n} \geq \min \left\{  \gro{x_n}{y_n}{t_n} , \gro {y_n}{z_n}{t_n}\right\} - \delta_n
\end{math}.
Taking the $\omega$-limit, we obtain 
\begin{math}
	\gro x z t \geq \min \left\{ \gro x y t , \gro y z t \right\} - \delta
\end{math}.
Thus $\limo X_n$ is $\delta$-hyperbolic.
\end{proof}

\begin{coro}
	Let $\omega$ be a non-principal ultra-filter and $(\delta_n)$ a sequence of non negative numbers such that $\limo \delta_n=0$.
	Let $\left(X_n, x^0_n\right)$ be a sequence of pointed geodesic spaces. If for all $n \in \N$, $X_n$ is $\delta_n$-hyperbolic, then  the limit space $\limo X_n$ is an $\R$-tree.
\end{coro}

\begin{proof}
	The $\omega$-limit of a sequence of geodesic spaces is still geodesic (cf. \cite{Pap96}). 
	It follows that $\limo X_n$ is a geodesic, 0-hyperbolic metric space. Hence $\limo X_n$ is an $\R$-tree.
\end{proof}

\begin{prop}
\label{reverse ultra limit of hyperbolic space}
	Let $\omega$ be a non-principal ultra-filter and $\delta$ a non negative number.
	Let $\left(X_n, x^0_n\right)$ be a sequence of pointed metric spaces  whose diameters are bounded. 
	If $\limo X_n$ is $\delta$-hyperbolic, then for all $\delta' > \delta$, $X_n$ is $\delta'$-hyperbolic \oasly.
	In particular there exists $n \in \N$ such that $X_n$ is $\delta'$-hyperbolic.
\end{prop}

\begin{proof}
	Assume that the proposition is false. 
	Then $\limo X_n$ is $\delta$-hyperbolic and nevertheless there exists $\delta' > \delta$ such that $X_n$ is \oasly\ not $\delta'$-hyperbolic
	Thus we can find four sequences $(x_n)$, $(y_n)$, $(z_n)$ and $(t_n)$ satisfying the following properties: 
	\begin{enumerate}
		\item for all $n \in \N$, $x_n, y_n, z_n, t_n \in X_n$,
		\item $\gro {x_n}{z_n}{t_n} < \min \left\{  \gro{x_n}{y_n}{t_n} , \gro {y_n}{z_n}{t_n}\right\} - \delta'$, \oasly.
	\end{enumerate}
	Since $\left(\diam (X_n)\right)$ is bounded, these four sequences define four points of $\limo X_n$, respectively $x$, $y$, $z$ and $t$. After taking the $\omega$-limit in the previous inequality we obtain 
	\begin{displaymath}
		\gro x z t \leq \min \left\{ \gro x y t , \gro y z t \right\} - \delta' < \min \left\{ \gro x y t , \gro y z t \right\} - \delta
	\end{displaymath}
	Hence $\limo X_n$ is not $\delta$-hyperbolic. Contradiction.
\end{proof}

\subsection{Quasi-convexity}

If $X$ is a geodesic space, there is another way to characterize the hyperbolicity using geodesic triangles. 
Let $\delta$ be a non negative number. 
A geodesic triangle is $\delta$-thin if each one of its sides is contained in the $\delta$-neighbourhood of the union of the two others.

\begin{prop}[cf. {\cite[Chap. 1 Prop. 3.6]{CooDelPap90}} or {\cite[Chap. 3 \S  2]{GhyHar90}}]
	Let $\delta$ be a non negative number. 
	Consider a geodesic space $X$.
	\begin{enumerate}
		\item If $X$ is $\delta$-hyperbolic, then every geodesic triangle of $X$ is $4\delta$-thin.
		\item If every geodesic triangle of $X$ is $\delta$-thin, then $X$ is $8\delta$-hyperbolic.
	\end{enumerate}
\end{prop}

\begin{coro}
\label{8 delta thin quadrilater}
	Let $x$, $x'$, $y$ and $y'$ be four points of a geodesic, $\delta$-hyperbolic space $X$.  
	If $u$ is a point of $\geo{x}{x'}$ such that $\distX{x}u >\distX{x}{y} + 8 \delta$ and $\distX{x'}u >\distX{x'}{y'} + 8 \delta$, then $u$ lies in the $8\delta$-neighbourhood of $\geo{y}{y'}$.
\end{coro}

\begin{proof}
Since the triangles $\left[ x, y, y'\right]$ and $\left[ x, x', y'\right]$ are $4\delta$-thin, we can find a point $v$ in $\geo{x}{y} \cup \geo{y}{y'} \cup \geo{y'}{x'}$ such that $\distX u v \leq 8 \delta$.  
We will show that $v \in \geo y{y'}$.
Suppose that $v \in \geo{x}{y}$ (the case $v \in \geo{x'}{y'}$ is symmetric). 
The triangle inequality gives
\begin{displaymath}
	\distX {x}u \leq \distX v u + \distX{x} v \leq \distX{x}{y} +8 \delta
\end{displaymath}
Contradiction. 
Consequently, $u$ lies in the $8\delta$-neighourhood of $\geo{y}{y'}$.

\end{proof}

\begin{defi}[Quasi-convexity]
	Let $\alpha$ be a non negative number. A subset $Y$ of a geodesic metric space $X$ is $\alpha$-quasi-convex if every geodesic between two points of $Y$ is contained in the $\alpha$-neighbourhood of $Y$.
\end{defi}

\nota We denote by $Y^{+ \alpha}$ the $\alpha$-neighbourhood of $Y$.

\begin{prop}[compare {\cite[Lemma 2.2.2]{DelGro08}}]
\label{estimation diameter quasi convex part}
	Let $\delta, \alpha \geq 0$. Let  $X$ be a geodesic, $\delta$-hyperbolic space. If $Y$ and $Z$ are two $\alpha$-quasi-convex subsets of $X$, then for all $A \geq 0$, we have
	\begin{displaymath}
		\diam\left(Y^{+A}\cap Z^{+A}\right) \leq \diam \left( Y^{+\alpha + 10 \delta} \cap Z^{+\alpha+10 \delta} \right) + 2A + 20 \delta
	\end{displaymath}
\end{prop}

\begin{proof}
	Let $x$ and $x'$ be two points of $Y^{+A}\cap Z^{+A}$ and assume that $\distX {x} {x'} \geq 2A + 20 \delta$. We choose
	\begin{enumerate}
		\item two points $t$ and $t'$  of $\geo {x} {x'}$ such that $\distX {x} {t} = \distX {x'} {t'} = A +10 \delta$,
		\item two points $y$ and $y'$ of $Y$ such that $\distX {x}{y}, \distX{x'}{y'} \leq A+\delta$.
	\end{enumerate}
Applying Corollary \ref{8 delta thin quadrilater}, $t$ belongs to the $8 \delta$-neighbourhood of $\geo{y}{y'}$. 
Since $Y$ is $\alpha$-quasi-convex, $\geo{y}{y'}$ lies in the $\alpha$-neighbourhood of $Y$. 
Hence $t$ belongs to $Y^{+\alpha+10\delta}$. 
We prove in the same way that $t$ belongs to $Z^{+ \alpha +10 \delta}$. 
The same fact holds for $t'$. 
Thus
	\begin{displaymath}
		\distX {x}{x'} - 2A - 20 \delta = \distX {t}{t'} \leq \diam \left( Y^{+\alpha + 10 \delta} \cap Z^{+\alpha+10 \delta} \right) 
	\end{displaymath}
	The above inequality is true for all $x, x' \in Y^{+A}\cap Z^{+A}$, hence
	\begin{displaymath}
		\diam\left(Y^{+A}\cap Z^{+A}\right) \leq \diam \left( Y^{+\alpha + 10 \delta} \cap Z^{+\alpha+10 \delta} \right) + 2A + 20 \delta
	\end{displaymath}
\end{proof}

\begin{coro}
\label{limit diameter intersection}
	Let $\omega$ be a non-principal ultra-filter, and $(\delta_n)$ a sequence of real numbers such that $\limo \delta_n = 0$.  For all $n \in \N$, let $\left(X_n, x^0_n\right)$ be a pointed, geodesic, $\delta_n$-hyperbolic space and $Y_n$, $Z_n$ two $10 \delta_n$-quasi-convex subsets of $X_n$. Let $X = \limo \left(X_n, x^0_n\right)$, and $Y= \limo Y_n$, $Z = \limo Z_n$. We have
	\begin{displaymath}
		\diam(Y\cap Z) \leq \limo \diam \left( Y_n^{+20 \delta_n}\cap Z_n^{+20 \delta_n}\right)
	\end{displaymath}
\end{coro}

\begin{proof}
	Let $x$ and $x'$ be two points of $Y\cap Z$. Since $x,x' \in Y$, we can find two sequences $(x_n)$ and $(x'_n)$ such that $x=\limo x_n$, $x'= \limo x'_n$ and $x_n, x'_n \in Y_n$ \oasly. Moreover $x$ and $x'$ belong to $Z$, thus if $A>0$ is given $x_n$ and $x'_n$ belong to $Z_n^{+A}$ \oasly. Using Proposition \ref{estimation diameter quasi convex part}, we  have
	\begin{displaymath}
		\distX {x_n}{x'_n} \leq \diam\left(Y_n^{+A}\cap Z_n^{+A}\right) \leq \diam\left(Y_n^{+20 \delta_n}\cap Z_n^{+20 \delta_n}\right) +2A + 20 \delta_n
	\end{displaymath}
	By taking the $\omega$-limit, we obtain $\distX x {x'} \leq \limo \diam\left(Y_n^{+20 \delta_n}\cap Z_n^{+20 \delta_n}\right) +2A$.
	This inequality is true for all $A>0$ and $x,x' \in Y\cap Z$, thus 
	\begin{displaymath}
		\diam(Y\cap Z) \leq \limo \diam \left( Y_n^{+20 \delta_n}\cap Z_n^{+20 \delta_n}\right)
	\end{displaymath}
\end{proof}

We need in Section \ref{part cone-off} a little stronger condition than the quasi-convexity.
\begin{defi}
	Let $X$ be a $\delta$-hyperbolic space. 
	A subset $Y$ of $X$ is strongly quasi-convex if for all $x,x' \in Y$ there exist $p,p' \in Y$ such that $\distX p x  \leq 10 \delta$ , $\distX {p'}{x'} \leq 10 \delta$ and the path $\geo x p \cup \geo p {p'} \cup \geo {p'}{x'}$ lies in $Y$.
\end{defi}

\rem Since any geodesic triangle of $X$ is $4\delta$-thin, any strongly quasi-convex space is $10 \delta$-quasi-convex.
Given a $10 \delta$-quasi-convex subset $Y$ of $X$, there is a way to find a subset of $X$, a little larger than $Y$, that is strongly quasi-convex. 
To that end, we define the cylinder of a subset.

\begin{defi}
Let $Y$ be a subset of a geodesic $\delta$-hyperbolic space $X$. 
The cylinder of $Y$, denoted by $\cyl(Y)$, is the set of all points which are in the $10 \delta$-neighbourhood of a geodesic of $X$ joining two points of $Y$.
\end{defi}

\begin{lemm}
\label{cylinder quasi-convex}
	Let $Y$ be a $10 \delta$-quasi-convex subset of a geodesic, $\delta$-hyperbolic space $X$. The set $\cyl(Y)$ is contained in  $Y^{+20 \delta}$ and is strongly quasi-convex.
\end{lemm}

\begin{proof}
	By definition of quasi-convexity, any geodesic joining two points of $Y$ lies in $Y^{+10 \delta}$. 
	It follows that $\cyl(Y) \subset Y^{+20 \delta}$.
	Let $x$ and $x'$ be two points of $\cyl(Y)$. By definition there exist two points of $Y$, $y_1$ and $y_2$ (\resp $y'_1$ and $y'_2$) such that $x$ (\resp $x'$) belongs to the $10 \delta$-neighbourhood of $\geo {y_1}{y_2}$  (\resp $\geo{y'_1}{y'_2}$). We denote $p$ and $p'$ the respective projections of $x$ and $x'$ on $\geo {y_1}{y_2}$ and $\geo{y'_1}{y'_2}$. 
	\begin{itemize}
		\item By construction, the geodesic segments $\geo x{p}$ and $\geo{p'}{x'}$ are contained in $\cyl(Y)$ and shorter than $10 \delta$.
		\item Since the triangles $\left[ y_2 , p, p'\right]$ and $\left[ y_2 , y'_2, p'\right]$ are $4 \delta$-thin, $\geo{p}{p'}$ stays in the $8\delta$-neighbourhood of $\geo{p}{y_2} \cup \geo{y_2}{y'_2} \cup \geo{y'_2}{p'}$. 
		However these segments are parts of geodesics between two points of $Y$. 
		Thus $\geo{p}{p'} \subset \cyl(Y)$.
	\end{itemize}
	Hence $\geo x p \cup \geo p {p'} \cup \geo {p'} {x'}$ lies in $\cyl(Y)$.
	
\end{proof}

\subsection{Asphericity}
\nota If $X$  is a simplicial complex, we denote by $X^{(k)}$ its $k$-skeleton.

\paragraph{} In this part we prove a version of the famous Rips' theorem: a hyperbolic simplicial complex which is  locally aspherical is globally aspherical.
Let $X$ be a metric space and $d$ a positive number. The Rips' polyhedron of $X$ denoted by $P_d(X)$ is a simplicial complex defined as follows. The simplices of $P_d(X)$ are the finite subsets of $X$ of diameter less than $d$. It is known (see \cite[Section 2.2]{Gro87}) that  if $X$ is geodesic $\delta$-hyperbolic, then for all $d \geq 4 \delta$ the polyhedron $P_d(X)$ is contractible. More precisely, we have the following proposition.

\begin{prop}[cf. {\cite[Chap. 5 Prop 1.1]{CooDelPap90}}]
\label{contraction rips}
	Let $X$ be a geodesic, $\delta$-hyperbolic space. Let $d \geq 4 \delta$ and $n \in \N$. The polyhedron $\ripsn {n+1} X$ is $n$-connected.
\end{prop}

Before studying the case of an arbitrary simplicial complex, we prove the following proposition.

\begin{prop}
\label{equivalence homotopy rips}
	Let $X$ be a $n$-dimensional simplicial complex. Let $d>1$. 
	Assume that  for all $r \leq 2(n+1)d$ and for all $x \in X$ the closed ball $\bar B(x,r)$ is contractible in $B(x,r+d)$. Then, there exist two maps:  $f : X \rightarrow \ripsn {n+1}{X^{(0)}}$  and  $g : \ripsn {n+1}{X^{(0)}} \rightarrow X$ such that $g \circ f$ is homotopic to $\id_X$.
\end{prop}

\begin{proof}
	In this proof we denote by $P$ the $(n+1)$-skeleton of the Rips' polyhedron $\rips {X^{(0)}}$.
	We define $f : X^{(0)} \rightarrow P$ by $f(x)=\{x\}$. Let $k\leq n$. If $\sigma$ is a $k$-simplex of $X$, its diameter is less than $d$. 
	Thus the set of its vertices defines a $k$-simplex of $P$. 
	Hence $f$ induces a simplicial map from $X$ to $P$. 
	We now define by induction a map $g : P \rightarrow X$.

		\paragraph{} First we define a map $g^{(0)} : P^{(0)} \rightarrow X$ by $g^{(0)} \left(\{x\}\right) = x$.
		\paragraph{} Assume now  that we have already defined a continuous map $g^{(k)} : P^{(k)} \rightarrow X$ with the following property:  for each $l$-simplex $\sigma$ of $P^{(k)}$, for every vertex $x$ of $\sigma$, $g^{(k)}(\sigma)$ is contained in $B(g^{(k)}(x), 2ld)$. 
		Let $\sigma$ be a $(k+1)$-simplex of $P^{(k+1)}$ whose faces are $\sigma_0, \dots , \sigma_{k+1}$. 
		Choose a vertex $x$ of $\sigma$. 
		The function $g^{(k)}$ maps the border $\partial \sigma = \bigcup_{i=0}^{k+1} \sigma_i$ onto a $k$-sphere of $X$ contained in $B(g^{(k)}(x), (2k+1)d)$. 
		However this sphere is contractible in $B(g^{(k)}(x), 2(k+1)d)$. 
		We  define $g^{(k+1)}(\sigma)$ by choosing a homotopy which contracts $g^{(k)}(\partial \sigma)$ to a point. 
		This defines a continuous map $g^{(k+1)} : P^{(k+1)} \rightarrow X$ which coincides with $g^{(k)}$ on $P^{(k)}$ and satisfies the following property: for all $l$-simplex $\sigma$ of $P^{(k+1)}$, for every vertex $x$ of $\sigma$, $g^{(k)}(\sigma)$ is contained in $B(g^{(k)}(x), 2ld)$. 
		We choose for $g$ the map $g^{(n+1)}$.
		
	\begin{lemm}
		For all $k \leq n+1$ there is a continuous map $H^{(k)} : X^{(k)} \times [0,1] \rightarrow X$ satisfying the following properties: 
		\begin{enumerate}
			\item $\restriction {H^{(k)}}{X^{(k)} \times \{0\}} = \id_{X^{(k)}}$ and   $\restriction {H^{(k)}}{X^{(k)}\times \{1\}} =\restriction {g \circ f}{X^{(k)}}$,
			\item for each $l$-simplex $\sigma$ of $X^{(k)}$, for every vertex $x$ of $\sigma$, $H^{(k)}(\sigma \times [0,1])$ is contained in $B(x,(2l+1)d)$.
		\end{enumerate}
	\end{lemm}
	
	\begin{proof}
		We prove this result by induction on $k$.
		The restriction of $g \circ f$ to $X^{(0)}$ is the identity, thus the proposition is obvious for the 0-skeleton.
		\paragraph{} Assume now that the lemma is true for $k \leq n$. 
		Consider a $(k+1)$-simplex $\sigma$ of $X^{(k+1)}$. 
		We chose a vertex $x$ of $\sigma$. 
		By definition of $g$ the set  $g \circ f (\sigma)$ is contained in $B(x,(2k+2)d)$. 
		Moreover, the induction assumption gives that $H^{(k)}(\partial \sigma \times [0,1]) \subset B(x,(2k+2)d)$. 
		Thus the subset $\sigma \cup H^{(k)}(\partial \sigma \times [0,1]) \cup g \circ f (\sigma)$ is a $(k+1)$-sphere of $X$ contained in $B(x,(2k+2)d)$. 
		This sphere is therefore contractible in $B(x,(2k+3)d)$. 
		By choosing a homotopy which contracts it to a point, we define a map $H^{(k+1)} : \sigma \times [0,1] \rightarrow X$ such that 
			\begin{enumerate}
				\item $\restriction {H^{(k+1)}}{\sigma\times \{0\}} = \id_\sigma$ and  $\restriction {H^{(k+1)}}{\sigma\times \{1\}} = \restriction {g \circ f}{\sigma}$,
				\item $\restriction {H^{(k+1)}}{ \partial \sigma\times [0,1]} = \restriction {H^{(k)}}{ \partial \sigma\times [0,1]}$,
				\item $H^{(k+1)} (\sigma \times [0,1]) \subset B(x, (2k+3)d)$.
			\end{enumerate}
			This defines a map $H^{(k+1)} : X^{(k+1)} \times [0,1] \rightarrow X$ which satisfies the properties of the lemma.
	\end{proof}
	
\paragraph{End of the proof of Proposition \ref{equivalence homotopy rips}}
The map $H^{(n+1)} : X\times [0,1] \rightarrow X$ is a homotopy between $g\circ f $ and $id_X$.
\end{proof}

\begin{theo}
\label{aspherical hyperbolic complex}
	Let $X$ be a $\delta$-hyperbolic, $n$-dimensional, simplicial complex. 
	Assume that for all $r \leq 8(n+1)\delta$ and for all $x \in X$, the ball $B(x,r)$ is homotopic to zero in $B(x,r+4 \delta)$. 
	Then, all homotopy groups of $X$ are trivial. Hence $X$ is contractible.
\end{theo}

\begin{proof}
 We fix $d = 4 \delta$. Using Proposition \ref{contraction rips}, the Rips' polyhedron $P= \ripsn {n+1}{X^{(0)}}$ is $n$-connected. Moreover, the fact that the small balls are aspherical gives two maps $f:X\rightarrow P$ and $g: P \rightarrow X$ such that $g \circ f$ is homotopic to $\id_X$. It follows that $X$ is also $n$-connected. Since $X$ is $n$-dimensional, all the higher homotopy groups of $X$ are trivial by Hurewicz Theorem.
\end{proof}

%% file: 2-cone.tex
\section{Cone over a metric space}
\label{part cone}

In this section we prove an asymptotic version of the Berestovskii's theorem concerning the hyperbolicity of a cone with a locally hyperbolic base.
From now on, we fix a positive number $r_0$ whose value will be made precise in Section \ref{part small cancellation}.

\subsection{Definition}

\begin{defi}
	Let $Y$ be a metric space. The cone over $Y$, denoted by $C(Y)$ is the quotient of  $Y\times \left[0, r_0 \right]$ by the equivalence relation defined as follows. Two points $(y_1,r_1)$ and $(y_2,r_2)$ are equivalent if $r_1=r_2=0$ or  $(y_1,r_1)=(y_2,r_2)$.
\end{defi}

\nota The equivalence class of $(y,0)$, called the vertex of the cone, is denoted $v_Y$ (or simply $v$). The equivalence class of any other point $(y,r)$ is still denoted by $(y,r)$.

\subsubsection*{Hyperbolic metric on a cone}

We define a metric on $C(Y)$ as M. Bridson and A. Haefliger do in \cite[Chap I.5]{BriHae99}. If $y$ and $y'$ are two points of $Y$, we consider the angle $\angle {y}{y'}$ defined by 
\begin{math}
	\angle{y}{y'} = \min\left\{ \pi, \frac{\distX{y}{y'}}{\sinh r_0}\right\}
\end{math}.

\begin{prop}[{\cite[Chap. I.5 Prop. 5.9]{BriHae99}}]
	The following formula defines a distance on the cone $C(Y,r_0)$.
	\begin{equation}
		\label{distance cone}
		\distX{(y,r)}{(y',r')} = \arccosh \Big( \cosh r \cosh r' - \sinh r \sinh r' \cos \angle {y}{y'}\Big)
	\end{equation}
\end{prop}

\rems 
\begin{itemize}
	\item The distance on the cone has the following interpretation. Given two points $(y,r)$ and $(y',r')$ of $C(Y)$,  the distance between them is the distance between two points of the hyperbolic disc respectively distant from the centre of $r$ and $r'$, such that the central angle between them is $\angle{y}{y'}$.
	\item It is important to notice that $\distX{(y,r)}{(y',r')} $ is a \textit{continuous} function of $y$, $y'$, $r$ and $r'$.
	\item The cone $C(Y)$ is the ball of centre $v$ and of radius $r_0$ of the space $C_{-1}\left( \frac Y{\sinh r_0}\right)$ defined in \cite[Chap. I.5]{BriHae99}.
\end{itemize}

\begin{prop}[{\cite[Chap. I.5 Prop. 5.10]{BriHae99}}]
\label{geodesic cone BH}
	Let $(y,r)$ and $(y',r')$ be two points of $C(Y,r_0)$.
	\begin{enumerate}
		\item If $r, r' >0$ and $\angle{y}{y'} < \pi$, then there is a bijection between the set of geodesics joining $y$ and $y'$ in $Y$ and the set of geodesics joining $(y,r)$ and $(y',r')$ in $C(Y)$.
		\item In all other cases, there is a unique geodesic joining $(y,r)$ and $(y',r')$.
	\end{enumerate}
\end{prop}

\exs 
\begin{enumerate}
	\item If $Y$ is a circle, endowed with its length metric, whose perimeter is $2\pi \sinh r_0$, then the cone $C(Y)$ is the hyperbolic disc of radius $r_0$.
	\item If $Y$ is isometric to a line, then $C(Y)\setminus \{ v \}$ is the universal cover of the punctured hyperbolic disc of radius $r_0$.
\end{enumerate}

\subsubsection*{Relation between the cone and its base}

In order to compare the cone $C(Y)$ and its base $Y$, we consider two maps: 

\begin{displaymath}
	\begin{array} {rccccrccc}
		\iota :	& Y	&\rightarrow	& C(Y)	& \quad	& p :	&	C(Y)\setminus \{v\}	& \rightarrow	&	Y \\
					& y	&\rightarrow	& (y,r_0)	&				&		& (y,r)							&	\rightarrow	& y \\
	\end{array}
\end{displaymath}

\begin{prop}
\label{projection distance cone}
	Let $(y,r)$ and $(y',r')$ be two points of $C(Y)$. Then
		\begin{displaymath}
			2\min\{r,r'\}\frac {\angle{y}{y'}}\pi \leq \distX {(y,r)}{(y',r')} \leq \distX {r}{r'} + \sqrt{\sinh r \sinh r'}\angle{y}{y'}
		\end{displaymath}
		In particular, let $x$ be a point of $C(Y)$ whose distance to $v$ is at least $\frac {r_0}2$.
		Then for every point $x' $ in the ball $B\left(x,\frac {r_0}3 \right)$ we have 
		\begin{math}
			\dist Y {p(x)}{p(x')} \leq \frac {3\pi \sinh r_0}{r_0} \dist {C(Y)}x {x'}
		\end{math}.
\end{prop}

\begin{proof}
	The inequalities follow from the facts that
	\begin{itemize}
		\item the map $t \rightarrow \arccosh\left(1+ a(1-\cos t)\right)$ is concave,
		\item for all $t \geq 0$, $\arccosh(a + t) \leq \arccosh(a) + \sqrt{2t}$.
	\end{itemize} 
	Consider now a point $x=(y,r)$ of  $C(Y)$ whose distance to $v$ is at least $\frac{r_0}2$. If $x'=(y',r')$ belongs to the ball $B\left(x,\frac {r_0}3 \right)$, then $\distX x {x'} < \frac{r_0}3< r < r+r'$. It follows  that $\angle y{y'} < \pi$. 
	Moreover $r' \geq \frac {r_0}6$.
	Using the previous inequality, we obtain
	\begin{math}
		\frac {r_0} {3 \pi \sinh r_0} \distX y {y'} \leq \frac {2\min\{r,r'\}} {\pi \sinh r_0} \distX y {y'} \leq \distX x {x'}
	\end{math}
\end{proof}

\begin{prop}
\label{function mu}
	Let $\mu  : \R^+ \rightarrow \R^+$, be the map defined by $\forall y,y' \in Y$, $\distX{\iota(y)}{\iota(y')} = \mu \left( \distX{y}{y'} \right)$. Then $\mu$ is non decreasing, continuous, concave map satisfying the following properties 
	\begin{enumerate}
		\item $\forall t, t' \geq 0$, $ \mu(t+t') \leq \mu(t) + \mu(t')$ (subadditivity),
		\item $\forall t \geq 0$, $\frac{2 r_0}{\pi \sinh r_0} \min\{\pi \sinh r_0, t \} \leq \mu(t) \leq t$.
	\end{enumerate}
\end{prop}

\begin{proof}
	By construction, the map $\mu$ is defined by 
	\begin{displaymath}
		\mu(t) = \arccosh \left( \cosh^2 r_0 - \sinh^2 r_0 \cos \left( \min \left\{ \pi, \frac t {\sinh r_0}\right\}\right)\right)
	\end{displaymath}
	The properties of $\mu$ follow from its concavity (cf Fig. \ref{graphe-mu}).
	
	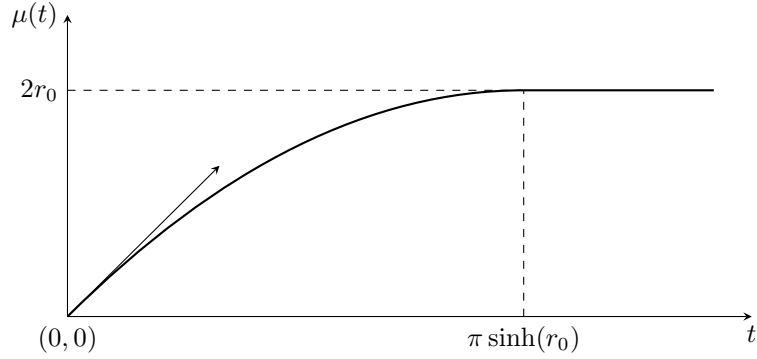
\begin{figure}[h]
	\centering
	\begin{tikzpicture}
		\draw [>=stealth, ->] (0,0) -- (9,0);
		\draw (9,0) node [below] {$t$};
		\draw [>=stealth, ->] (0,0) -- (0,4);
		\draw (0,4) node [left] {$\mu(t)$};
		\draw (0,0) node [left, below] {$(0,0)$};
		\draw [dashed] (6,0) -- (6,3);
		\draw (6,0) node[below] {$\pi \sinh(r_0)$};
		\draw [dashed] (0,3) -- (6,3);
		\draw (0,3) node [left] {$2 r_0$};
		\draw [thick] (6,3) -- (8.5,3);
		\draw [thick] (0,0) .. controls +(2,2) and +(-2,0) .. (6,3);
		\draw [thin,>=stealth, ->] (0,0) -- (2,2);
	\end{tikzpicture}
	\caption{Graph of $\mu$.}
	\label{graphe-mu}
	\end{figure}
	
\end{proof}

\subsection{Cone and hyperbolicity}

\begin{lemm}
\label{inversion cone ultra limit}
	Let $\omega$ be a non-principal ultra-filter.  Let  $\left( Y_n, y^0_n \right)$ be a sequence of pointed metric spaces. 
	We assume that the sequence $\left( \diam( Y_n)\right)$ is bounded.
	The spaces $C\left( \limo \left(Y_n, y^0_n \right)\right)$ and  $\limo \left( C(Y_n), \iota_n(y^0_n) \right)$ are isometric.
\end{lemm}

\begin{proof}
	Denote by $Y$ the limit space $\limo \left(Y_n, y^0_n \right)$. 
	We define  a map $f : C(Y) \rightarrow \limo C(Y_n)$ by  $f(\limo y_n, r) = \limo (y_n,r)$. Since Formula (\ref{distance cone}), giving the distance in a cone, is continuous, the map $f$ preserves the distances.  Consider now a point $x= \limo (y_n, r_n)$ of $\limo C(Y_n)$. By assumption, the sequences $(r_n)$  and $\left(\distX{y^0_n}{y_n}\right)$ are bounded.  Thus we may consider the real number $r= \limo r_n$ and the point $y =\limo y_n$ of $Y$.
	Furthermore, 
	\begin{displaymath}
		\distX x {f(y,r)} = \limo \distX {(y_n,r_n)}{(y_n,r)} = \limo \distX {r_n}r = 0
	\end{displaymath} 
	It follows that $f(y,r)=x$. 
	Hence $f$ is onto.
\end{proof}

\begin{lemm}
\label{cone over R tree}
	Let $Y$ be a metric space. If every ball of radius $\pi \sinh r_0$ of $Y$ is an $\R$-tree, then the cone $C(Y)$ is CAT(-1). In particular this cone is $\ln 3$-hyperbolic.
\end{lemm}

\begin{proof}
	Let $T$ be a geodesic triangle  of $Y$ whose perimeter is less than $2\pi \sinh r_0$.
	It is contained in a ball of radius $\pi \sinh r_0$. 
	It follows that $T$ is 0-thin. 
	Consequently the rescaled space $\frac Y {\sinh r_0}$ is CAT(1) (cf \cite{BriHae99}). 
	Using a Berestovskii's theorem (cf. \cite[Chap. II.3 Th. 3.14]{BriHae99}), the cone $C(Y)$ is CAT(-1). 
	In particular it is $\ln 3$-hyperbolic.
\end{proof}

\begin{prop}
\label{cone over hyperbolic space}
	Let $\epsilon >0$. 
	There exists $\delta >0$ satisfying the following property. Let $X$  be a geodesic space such that each ball of radius $\pi \sinh r_0$ of $X$ is $\delta$-hyperbolic. 
	Let $Y$ be a $10 \delta$-quasi-convex subset of $X$. The cone $C(Y)$ is $(\ln3 + \epsilon)$-hyperbolic.
\end{prop}

\begin{proof}
	Assume that the proposition is false. For all $n \in \N$, we can find 
	\begin{itemize}
		\item a geodesic space $X_n$, whose balls of radius $\pi \sinh r_0$ are $\delta_n$-hyperbolic, with $\delta_n = o(1)$,
		\item a $10 \delta_n$-quasi-convex subset $Y_n$ of $X_n$,
	\end{itemize}
	such that the cone $C(Y_n)$ is not $(\ln3 + \epsilon)$-hyperbolic. We denote by $\tilde X_n$, the space $X_n$ endowed with the following metric
	\begin{displaymath}
		\dist {\tilde X_n}x{x'} = \min \left\{ \pi \sinh r_0, \dist{X_n}x{x'}\right\}
	\end{displaymath} 
	The set $Y_n$ viewed as a subspace of $\tilde X_n$ is denoted by $\tilde Y_n$.
	We choose a non-principal ultra-filter $\omega$, the limit space $\tilde X = \limo \tilde X_n$ and the subspace $\tilde Y = \limo \tilde Y_n$.
	Each ball of radius $\pi \sinh r_0$ of $\tilde Y_n$ is $\delta_n$-hyperbolic. 
	Hence each ball of radius $\pi \sinh r_0$ of $\tilde Y$ is 0-hyperbolic.  
	Moreover $Y_n$ is $10 \delta_n$ quasi-convex. 
	It follows that, for all $y \in \tilde Y$, the set $\tilde Y \cap B(y, \pi \sinh r_0)$ is  an $\R$-tree. 
	Using Lemma \ref{cone over R tree}, $C(\tilde Y)$ is $\ln 3$-hyperbolic. 
	Moreover, the diameter of $\tilde Y_n$ is uniformly bounded. 
	Hence Lemma \ref{inversion cone ultra limit} tells us that $C(\tilde Y)$ and $\limo C(\tilde Y_n)$ are isometric. 
	Thus there exists $n \in \N$ such that $C(\tilde Y_n)$ is $(\ln 3+\epsilon)$-hyperbolic. 
	However $C(Y_n)$ and $C(\tilde Y_n)$ are isometric.
	Contradiction.
\end{proof}

\subsection{Group acting on a cone}

\begin{defi}
	Let $X$ be a metric space and $G$ a group acting on $X$ by isometries. For all $g \in G$ the translation length of $g$, denoted by $[g]_X$ (or simply $[g]$) is 
	\begin{displaymath}
		[g]_X = \inf_{x \in X} \distX{x}{gx}
	\end{displaymath}
	The injectivity radius of $G$ on $X$ is $\rinj GX = \inf_{g \in G\setminus \{1\}} [g]_X$.
\end{defi}

Let $Y$ be a metric space. We fix a group $H$ acting by isometries on $Y$. 
We assume that this action is proper, that is for all $y \in Y$ there exists $r>0$ such that the set $\left\{ h \in H / h. B(y,r) \cap B(y,r) \neq \emptyset \right\}$ is finite. 
We denote by $\bar Y$ the quotient $Y/H$ and by $\bar y$ the image in $\bar Y$ of a point $y \in Y$. 
Since $H$ acts properly on $Y$, the quotient $\bar Y$ may be endowed with a metric defined by
\begin{math}
	\dist{\bar Y} {\bar y} {\bar y'} = \inf_{h \in H} \dist Y {y} {hy'} 
\end{math}

\paragraph{} We extend the action of $H$ to the cone $C(Y)$ by homogeneity: 
If $x=(y,r)$ is a point of $C(Y)$ and $h$ an element of $H$, then $hx$ is defined by $hx = (hy,r)$. 
The group $H$ acts by isometries on $C(Y)$. 
Note that, if $Y$ is not compact, this action is no more proper. 
However the relation
\begin{math}
	\distX {\bar x} {\bar x'} = \inf_{h \in H} \dist {C(Y)} {x} {hx'} 
\end{math}
still defines a metric on the quotient $C(Y)/H$.  Moreover $C(Y/H)$ and $C(Y)/H$ are isometric.

\begin{theo}[First hyperbolicity theorem]
\label{first hyperbolicity theorem}
	Let $\epsilon >0$. 
	There exists $\delta >0$ satisfying the following property. 
	Let $X$ be a $\delta$-hyperbolic, geodesic space and $Y$ a $10 \delta$-quasi-convex subset of $X$. 
	Assume that $H$ is a group acting by isometries on $X$ such that  $H$ stabilizes $Y$. If $\rinj HY > 2 \pi \sinh r_0$, then the space $C(Y)/H$ is $(\ln 3 + \epsilon)$-hyperbolic.
\end{theo}

\begin{proof}
	We consider the constant $\delta >0$ given by Proposition \ref{cone over hyperbolic space}.
	Let $X$ be a $\delta$-hyperbolic, geodesic space and $Y$ a $10 \delta$-quasi-convex subset of $X$. 
	Assume that $H$ is a group acting by isometries on $X$ such that  $H$ stabilizes $Y$ and $\rinj HY > 2 \pi \sinh r_0$. 
	The spaces $X/H$ and $Y/H$ satisfy the assumptions of Proposition \ref{cone over hyperbolic space}. It follows that $C(Y/H)$, which is isometric to $C(Y)/H$, is $(\ln 3 + \epsilon)$-hyperbolic.
\end{proof}

\subsection{Contracting balls in a cone}

\cmath{Section \ref{contracting ball in a cone} is new. It explains how contracting small balls in a cone. 
This is used in order to prove Theorem \ref{asphericity universal cover}.}

\label{contracting ball in a cone}
	\paragraph{}
	In this section we assume that $Y$ is a proper metric space.
	The cone over $Y$ is contractible, nevertheless it is not necessarily locally contractible.
	To avoid this problem, we consider the following geometric assumption.
	
	\paragraph{Condition H(l): } For all $y \in Y$, for all $r \in \R_+$, there exists a homotopy 
	\begin{math}
		h : \bar B(y, r) \times [0,1] \rightarrow Y
	\end{math}
	contracting the closed ball $\bar B(y,r) $ to $\{  y \}$ with the following property: for all $y' \in \bar B(y, r)$, for all $t \in [0,1]$, $\distX y{h(y',t)} \leq \distX y{y'} + l$.
	
	\begin{prop}
	\label{contracting ball contained in a cone}
	Let $x$ be a point of the cone $C(Y)$ and $r_1 \in \R^+$.
	If $Y$ satisfies the condition H(l), then the closed ball $\bar B(x,r_1)$ is contractible in $B(x,r_1+l)$
	\end{prop}
	
	\begin{proof}
		We denote by $(y,r)$ the point $x$ and by $\bar B$ the closed ball $\bar B(x,r_1)$.
		We distinguish two cases.
		\subparagraph{Case 1.} If $r_1\geq r$, then the vertex of the cone $v$ belongs to $\bar B$.
		We consider the following homotopy.
		\begin{displaymath}
			\begin{array}{rccc}
				H : 	& \bar B \times [0,1] 	& \rightarrow 	& C(Y)\\
						& \Big((y',r'), t \Big)		& \rightarrow 	& \Big( y', (1-t)r'\Big)\\
			\end{array}
		\end{displaymath}
		
			It contracts the ball $\bar B$ to the vertex $v$.
			Let $x' = (y',r')$ be a point of $\bar B$.
			The metric in $\HP 2$ is convex. 
			Since the metric on the cone is modelled on the one on $\HP 2$, we have for all $t \in [0,1]$,
			\begin{displaymath}
				\distX x{H(x',t)} = \distX x{\Big(y',(1-t)r'\Big)} \leq \max \left\{ \distX x{x'} , \distX xv\right\} \leq r_1
			\end{displaymath}
			Thus $H$ maps to $\bar B$.
			
		\subparagraph{Case 2.} We assume now that $r_1 <r$.
		It follows that for all $(y',r') \in  \bar B$, $\distX y{y'} \leq \pi \sinh r_0$.
		By assumption, there exists a homotopy $h : \bar B (y, \pi \sinh r_0) \times [0,1] \rightarrow Y$ contracting $\bar B(y, \pi \sinh r_0)$ to $\{y\}$ such that for all $y' \in \bar B(y, \pi \sinh r_0)$, for all $t \in [0,1]$, $\distX y{h(y',t)} \leq \distX y{y'} + l$.
		We consider the following map.
		\begin{displaymath}
			\begin{array}{rccc}
				H : 	& \bar B \times [0,1] 	& \rightarrow 	& C(Y) \\
						& \Big( (y',r'),t \Big)		& \rightarrow 	& \Big (h(y',t),r' \Big) \\
			\end{array}
		\end{displaymath}
		It contracts $\bar B$ to $\{y\} \times [r-r_1 , \min \{r_0, r+r_1\}]$.
		Let $x' = (y',r')$ be a point of $\bar B$ and $t \in [0,1]$.
		By assumption, we have $\theta(y,h(y',t)) \leq \theta(y,y') + \alpha$, where $\alpha = \min\left\{ \frac l{\sinh r_0}, \pi - \theta(y,y')\right\}$.
		The distances between $x$, $x'$ and $H(x',t)$ may be viewed in $\HP 2$.
		Thanks to the triangle inequality, we have:
		\begin{eqnarray*}
			\distX x{H(x',t)}
			& = & \arccosh \left( \cosh r \cosh r' - \sinh r \sinh r' \cos \theta(y,h(y',t))\right) \\
			& \leq & \arccosh \left( \cosh r \cosh r' - \sinh r \sinh r' \cos \theta(y,y)\right) \\
			&		  &+ \arccosh \left( \cosh^2 r' - \sinh^2 r' \cos \alpha\right) \\
			& \leq & \distX x{x'} + \alpha \sinh r'  \leq r_1 +l \\
		\end{eqnarray*}
		Consequently, $H$ maps to $B(x,r_1 +l)$.
		We conclude by noticing that $\{ y \} \times [r-r_1 , \min \{r_0, r+r_1\}]$ is contractible in $B(x,r_1 +l)$.
	\end{proof}
	
	The next lemma explains how to deformation retract a ball of the cone onto its base.
	\begin{prop}
	\label{contracting ball to the base in the cone}
		Let $x=(y,r)$ be a point of the cone $C(Y)$ and $r_1 \in ]r_0-r , r[$.
		Assume that $Y$ satisfies the condition H(l).
		Then there exists a homotopy $H : \bar B(x,r_1) \times [0,1] \rightarrow \bar B (x, r_1 + 2l)$  which contracts $\bar B(x,r_1)$ to a subset of $\iota(Y)$ and such that for all $x' \in \bar B(x,r_1) \cap \iota (Y)$, for all $t \in [0,1]$, $H(x',t) = x'$.
	\end{prop}
	
	\rem 
	The first idea to prove this proposition is to project the cone onto its base, using the map $p$.
	Nevertheless, this homotopy does not stay in a neighbourhood of $\bar B(x,r_1)$.
	This problem can be observed on Figure \ref{shape of a ball} which represents the shape of $\bar B(x,r_1)$.
	In order to get around this difficulty, we have to proceed in two steps.
	At first, we contract the ball horizontally, using a homotopy of $Y$, then we project it onto the base.
	
	\begin{figure}[h]
	\begin{center}
	\begin{tikzpicture}
		\draw [>=stealth, ->] (0,0) -- (5,0);
		\draw (5,0) node [above] {$\distX y{y'}$};
		\draw [>=stealth, ->] (0,0) -- (0,-5);
		\draw (0,-5) node [left] {$r'$};
		\draw (0,0) node [left, above] {$(0,0)$};
	 	\draw [dashed] (0,-4) -- (4.5,-4);
	 	\draw (0,-4) node [left] {$r_0$};
	 	\draw[fill=gray!20, thick] 
	 						(0,-0.5) 	.. controls +(2,0) and +(0,0.5) .. (4,-1)
	 										.. controls +(0,-0.5) and +(0.7,2.5) .. (0.5,-4)
	 										-- (0,-4);
	 	\draw (1,-1.5) node {$\bar B(x,r_1)$};
	\end{tikzpicture} \newline
	 The points $(y',r')$ which belong to $\bar B(x,r_1)$ lie in the gray part.
	\end{center}
	\caption{Shape of the ball $\bar B(x,r_1)$.}
	\label{shape of a ball}
	\end{figure}

	\begin{proof}	
		We denote by $\bar B$ the closed ball $\bar B(x, r_1)$.
		Since $r_1 < r$, for all points $x'=(y',r')$ of $\bar B$, we have $\distX y{y'} \leq \pi \sinh r_0$.
		By assumption there exists a homotopy $h : \bar B(y, \pi \sinh r_0) \times [0,1] \rightarrow Y$ contracting $\bar B(y, \pi \sinh r_0)$ to $\{y \}$ such that for all $y' \in \bar B (y, \pi \sinh r_0)$,  for all $t \in [0,1]$, $\distX y{h(y',t)} \leq \distX y{y'} + l$.
		
		\paragraph{} Let $x'=(y',r')$ be a point of $C(Y)$.
		There exists a continuous function $L(r')  = \arccos\left( \frac{\cosh r' \cosh r - \cosh r_1}{\sinh r' \sinh r}\right)$ such that $x'$ belongs to $\bar B$ if and only if $\distX y{y'} \leq L(r')$.
		Since $r_1 > r_0-r$, $L(r_0)$ is positive.
		By continuity, there exists $r_2 < r_0$ such that for all $r' \in  ]r_2, r_0[$, $L(r') \leq L(r_0) + l$.
		Since $h$ is continuous on a compact set, there exists $t_0$ such that for all $y' \in \bar B(y, \pi \sinh r_0)$, $\distX y{h(y',t_0)} \leq L(r_0)$.
		We consider now the following map.
		\begin{displaymath}
			\begin{array}{rcccl}
				H : 	& \bar B \times [0,1] 	& \rightarrow 	& C(Y)																					& \\
						& \Big((y',r'),t\Big)		& \rightarrow 	& \big( h\left(y', tt_0 \right), r'\big) 									& \text{ if } r' \leq r_2 \\
						& \Big((y',r'),t\Big)		& \rightarrow 	& \left( h\left(y', \frac{r_0-r'}{r_0-r_2}tt_0\right), r'\right) 	& \text{ if } r' > r_2 \\
			\end{array}
		\end{displaymath}
		The map $H$ is continuous.
		Furthermore, for all $x' \in \bar B \cap \iota(Y)$, for all $t \in [0,1]$, $H(x',t) = x'$.
		As in the previous proposition, we prove that $H$ maps to $B(x,r_1 + l)$.
		We denote by $E$ the image of $\bar B \times \{1\}$ by $H$.
		
		\begin{lemm}
		\label{projection of E on the base}
			The set $E$, image by $H$ of $\bar B \times \{1\}$, has the following property.
			For all $x'=(y',r') \in C(Y)$, if $x'$ belongs to $E$, then $\distX y{y'} \leq L(r_0) + 2l$.
		\end{lemm}
		
		\begin{proof}
			Let $x'=(y',r')$ be a point of $\bar B$.
			If $r' \leq r_2$, then $H(x',1) = (h(y',t_0),r')$.
			By construction of $t_0$, we have $\distX y{h(y',t_0)} \leq L(r_0)$.
			If $r' > r_2$, then $H(x',1) = \left( h \left(y', \frac{r_0-r'}{r_0-r_2} t_0 \right) , r' \right)$.
			By definition of $r_2$ and $h$, we have the following inequalities.
			\begin{displaymath}
				\distX y{h \left(y', \frac{r_0-r'}{r_0-r_2} t_0 \right) } \leq \distX y{y'} + l \leq L(r_0) + 2l
			\end{displaymath}
		\end{proof}
		
		\paragraph{End of the proof of Proposition \ref{contracting ball to the base in the cone}}
		We now consider a second homotopy.
		\begin{displaymath}
			\begin{array}{rccc}
				H' : 	& E \times [0,1] 		& \rightarrow 	& C(Y) \\
						& \Big( (y',r'),t \Big) 	& \rightarrow 	& \left(y',(1-t)r' + tr_0 \right)  
			\end{array}
		\end{displaymath}
		The map $H'$ contracts $E$ to a subset of $\iota(Y)$.
		Let $x'=(y',r')$ be a point of $E \subset B(x,r_1+l)$.
		By Lemma \ref{projection of E on the base}, $\distX y{y'} \leq L(r_0) +2l$.
		Thus $x'$ and $\iota(y')$ are points of $B(x,r_1 + 2l)$.
		The metric in $\HP 2$ is convex. 
		Since the metric on the cone is modelled on the one on $\HP 2$, we have for all $t \in [0,1]$,
		\begin{eqnarray*}
			\distX x{H'(x',t)} 
			& = & \distX x{\Big(y',(1-t)r'+ tr_0\Big)} \\
			& \leq & \max \left\{ \distX x{x'} , \distX x{\iota(y')}\right\} \leq r_1 + 2l \\
		\end{eqnarray*}
		Thus $H'$ maps to $B(x, r_1 + 2l)$.
		Applying successively $H$ and $H'$ provides the homotopy of the proposition.
	\end{proof}

%% file: 3-cone-off.tex
\section{Cone-off over a metric space}
\label{part cone-off}

The goal of this part is to study the large scale geometry of the cone-off. 
We give a detailed proof that under some small cancellation assumptions, the cone-off of a hyperbolic space is locally hyperbolic.
We recall that $r_0$ is a fixed positive number whose value will be made precise in Section \ref{part small cancellation}.
\subsection{Definition}
\label{part definition cone-off}

Let  $X$ be a metric space and  $Y= (Y_i)_{i \in I}$ a collection of  subsets of $X$. For each $i\in I$, we consider the following objects.
\begin{itemize}
	\item $C_i$ is the cone $C(Y_i)$ and $v_i$ its vertex.
	\item $\iota_i : Y_i \rightarrow C_i$ and $p_i : C_i \setminus \{v_i\} \rightarrow Y_i$ are the maps between the cone and its base defined in the previous part.
\end{itemize}

\begin{defi}[Cone-off over a metric space]
	The cone-off over $X$ relatively to $Y$ is the space obtained by gluing each cone $C_i$ on $X$ along $Y_i$ according to $\iota_i$. We denote it by $\dot X(Y)$ (or simply $\dot X$).
\end{defi}

\subsubsection*{Metric on the cone-off}

\paragraph{} In this paragraph, we define a metric on the cone-off such that its restriction to a cone is the metric previously defined. 
To this end, one defines the metric as the lower bound of the length of chains joining two points. 
In this way, we obtain a pseudo-metric. 
The goal here is to prove that it is also positive.

\paragraph{}

We endow $X \sqcup \left( \bigsqcup_{i\in I} C_i\right)$ with the metric induced by $\distV{X}$ and $\distV{C_i}$.
Let $x$ and $x'$ be two points of $\dot X$.
The quantity $\distSC x{x'}$ is the minimal distance between two points of $X \sqcup \left( \bigsqcup_{i\in I} C_i\right)$ whose image in $\dot X$ are respectively $x$ and $x'$.
The value of $\distSC x{x'}$ is different whether $x$ and $x'$ both belong to a cone $C_i$ or not. 
The three possible cases are presented in the next lemma.

\paragraph{}Recall that $\mu$ is the function defined in Proposition \ref{function mu} by
\begin{displaymath}
		\mu(t) = \arccosh \left( \cosh^2 r_0 - \sinh^2 r_0 \cos \left( \min \left\{ \pi, \frac t {\sinh r_0}\right\}\right)\right)
	\end{displaymath}
It has the following interpretation.
If $y$ and $y'$ are two points of $Y_i$, then the distance between $(y,r_0)$ and $(y',r_0)$ in $C_i$ is $\mu\left(\dist X y {y'}\right)$.

\begin{lemm}
\label{some facts about distSC}
Let $x$ and $x'$ be two points of $\dot X$.
\begin{enumerate}
	\item If  there is $i \in I$ such that $x,x' \in C_i$, then $\distSC x{x'} = \dist{C_i}x{x'}$. 
	In particular, if $x,x' \in Y_i$, then $\distSC  x{x'} = \mu(\dist X x{x'})$.
	\item If $x,x' \in X$, but there is no $i \in I$ such that $x,x' \in Y_i$, then $\distSC  x{x'} = \dist X x{x'}$.
	\item In all other cases, we have $\distSC x{x'} = + \infty$.
\end{enumerate}
In particular, for all $x,x' \in X$, $\distSC x{x'} \geq \mu(\dist Xx{x'})$.
\end{lemm}

The quantity $\distSC x{x'}$ does not define a metric. 
It does not indeed satisfy the triangle inequality.
That is why we consider chains of points.

\begin{defi}
	Let $x$ and $x'$ be two points of $\dot X$. 
	\begin{itemize}
		\item A chain between $x$ and $x'$ is a finite sequence $C=(z_1,\dots, z_n)$ of points of $\dot X$, such that $z_1=x$ and $z_m=x'$. 
		\newline
		Its length is 
		\begin{math}
			l(C) = \sum_{j=1}^{m-1} \distSC{z_j}{z_{j+1}}
		\end{math}.
		\item Furthermore we define
		\begin{displaymath}
			\dist{\dot X} x{x'}  = \inf \left\{ l(C)/ C \text{ a chain between $x$ and $x'$}\right\}
		\end{displaymath}
	\end{itemize}
\end{defi}

Obviously, $|\ . \ |_{\dot X}$ is a pseudo-metric.

\begin{lemm}
\label{chain in X}
	Let $x$ and $x'$ be two points of $\dot X$. For all $\epsilon >0$ there is a chain $C=(z_1,\dots, z_m)$ between $x$ and $x'$ satisfying the following.
	\begin{enumerate}
	 	\item $l(C) \leq \dist{\dot X}x{x'} + \epsilon$
	 	\item For all $j \in \intval 2 {m-1}$, $z_j \in X$.
	\end{enumerate}
\end{lemm}

\begin{proof}
	Let $\epsilon >0$. By definition there exists $C = (z_1, \dots z_m)$ a chain between $x$ and $x'$ such that $l(C) \leq \dist{\dot X} x{x'} +\epsilon$.
	Assume that there is $j \in \intval 2 {m-1}$ such that $z_j$ does not belong to $X$. 
	It follows that $z_j$ is strictly contained in a cone, that is there exists $i \in I$ such that  $z_j \in C_i \setminus \iota_i(Y_i)$. 
	In particular there is only one point of $\left(\bigsqcup_{i\in I} C_i \right) \sqcup X$ whose image in $\dot X$ is $z_j$. Using the triangle inequality in $C_i$, we have
\begin{math}
	\distSC {z_{j-1}}{z_{j+1}} \leq \distSC{z_j}{z_{j+1}} + \distSC{z_{j-1}}{z_j}
\end{math}.
Thus the sequence $C'$, obtained by removing the point $z_j$ from $C$, is a chain between $x$ and $x'$ shorter than $C$. 
Hence after removing all points of $C$, which are not in $X$, we obtain a chain satisfying  the properties of the lemma.
\end{proof}

\begin{lemm}
\label{distance coincide on the cone}
	For all $i \in I$, $|\ . \ |_{C_i}$ and $|\ . \ |_{\dot X}$ locally coincide: 
	Let $x = (y,r)$ be a point of  $C_i\setminus \iota_i(Y_i)$.
	If $x'$ is a point of $\dot X$ such that $\dist{\dot X} x {x'} \leq \frac 14 |r_0 - r|$, then $x' \in C_i$ and $\dist{\dot X} x {x'} = \dist{C_i} x {x'}$.
\end{lemm}

\begin{proof}
	Let  $i \in I$ and $x=(y,r)$ be a point of $C_i\setminus \iota_i(Y_i)$. Since $x \notin \iota_i(Y_i)$, $r_0-r >0$. 
	Let $x'$ be a point of $\dot X$ such that  $\dist{\dot X} x {x'} \leq \frac 14 (r_0-r)$. 
	Let $\eta \in \left] 0 , \frac 14 (r_0-r) \right[$. 
	Using the previous lemma, there is a chain $C=(z_1,\dots, z_m)$ between $x$ and $x'$ such that $l(C) \leq \dist{\dot X} x {x'} + \eta$ and for all $j \in \intval 2 {m-1}$, $z_j \in X$. 
	Assume that $m \geq 3$. 
	Since $x \in C_i \setminus \iota_i(Y_i)$, we have
	\begin{displaymath}
		r_0 - r \leq \distSC {z_1}{z_2} \leq l(C) \leq \dist{\dot X} x {x'} + \eta \leq \frac 12 (r_0-r)
	\end{displaymath}
	Contradiction.
	Thus $m=2$ and $\distSC x {x'} = l(C) \leq \frac12 (r_0-r)$.  
	Consequently $x'$ belongs to $C_i$ and $\distSC x {x'} = \dist{C_i}x{x'}$.
	Hence for all $\eta \in \left] 0 , \frac 14 (r_0-r) \right[$, we have
	\begin{displaymath}
		\dist{\dot X} x {x'} \leq \dist{C_i} x {x'} \leq \dist{\dot X} x {x'} + \eta
	\end{displaymath}
	It follows that $\dist{\dot X} x {x'} = \dist{C_i} x {x'}$.
\end{proof}

\begin{lemm}
\label{estimation distance in the base}
	For all $x,x' \in X$, we have $\dist{\dot X} x{x'} \geq \mu(\dist X x{x'})$.
\end{lemm}

\begin{proof}
	Let $\epsilon >0$. 
	Using Lemma \ref{chain in X}, there exists a chain $C=(z_1,\dots, z_m)$ between $x$ and $x'$ such that $l(C) \leq \dist{\dot X} x {x'} + \epsilon$ and for all $j \in \intval 1 m$, $z_j \in X$. The subadditivity of $\mu$ gives
	\begin{displaymath}
		\mu(\dist X x {x'}) \leq \sum_{j=1}^{m-1} \mu(\dist X {z_j}{z_{j+1}}) \leq \sum_{j=1}^{m-1} \distSC {z_j}{z_{j+1}} = l(C)
	\end{displaymath}
	Thus for all $\epsilon >0$, we have $\mu(\dist X x {x'}) \leq \dist{\dot X} x {x'} + \epsilon$. It follows that $\dist{\dot X} x{x'} \geq \mu(\dist X x{x'})$.
\end{proof}

\begin{prop}
	$|\ . \ |_{\dot X}$ defines a metric on $\dot X$.
\end{prop}

\begin{proof}
The only point to prove is the positivity of $|\ . \ |_{\dot X}$. Consider  two points $x$ and $x'$ of $\dot X$ such that $\dist{\dot X} x{x'} = 0$. There are two cases.
\begin{enumerate}
\item If there is $i \in I$ such that $x \in C_i\setminus \iota_i(Y_i)$ or  $x' \in C_i\setminus \iota_i(Y_i)$, then, using Lemma \ref{distance coincide on the cone}, $x$ and $x'$ both belong to $C_i$. Moreover $\dist{C_i}x {x'} = \dist{\dot X} x {x'} = 0$. Thus $x=x'$.
\item If $x$ and $x'$ are both elements of $X$, then thanks to Lemma  \ref{estimation distance in the base} we have $\mu(\dist X x {x'}) \leq \dist{\dot X} x {x'} = 0$. Hence $\dist X x {x'} = 0$. It follows that $x=x'$.
\end{enumerate}
\end{proof}

\subsubsection*{Projection of the cone-off on its base}

We consider a map $p$ from $\dot X \setminus \{v_i, i\in I\}$ onto $X$ whose restriction to a cone $C_i\setminus\{v_i\}$ is $p_i$  and whose restriction to $X$ is the identity.

\begin{prop}
\label{comparison distance base cone-off}
	Consider a point $x$ of $\dot X$ such that the distance between $x$ and any vertex of $\dot X$ is at least $\frac{r_0}2$. 
	For all $x' \in B\left(x,\frac{r_0}3 \right)$, we have $\dist X {p(x)}{p(x')} \leq \frac {3 \pi \sinh r_0} {r_0} \dist {\dot X} x {x'}$.
\end{prop}

\begin{proof}
Let $\epsilon \in ]0 ,\frac{r_0}2[$. 
Consider a point $x'$ of $B\left(x,\frac{r_0}3 \right)$. 
Using Lemma \ref{chain in X}, there exists a chain $C = (z_1, \dots, z_m)$ such that for all $j$ that belongs to $\intval 2 {m-1}$, $z_j \in X$ and $l(C) \leq \dist {\dot X} x {x'} + \epsilon \leq r_0$.
We chose $j \in \{2, \dots , m-1\}$.
Lemma \ref{function mu} gives 
\begin{eqnarray*}
	r_0 \geq l(C) \geq \distSC {z_j}{z_{j+1}} & \geq & \mu (\dist X {z_j}{z_{j+1}} ) \\
	&\geq & \frac {2 r_0}{\pi \sinh r_0} \min \left\{\pi \sinh r_0, \dist X {z_j}{z_{j+1}} \right\}\\
\end{eqnarray*}
Thus $\dist X  {p(z_j)}{p(z_{j+1})} = \dist X {z_j}{z_{j+1}} \leq \frac {3\pi \sinh r_0} {r_0} \distSC {z_j}{z_{j+1}}$

If $x=z_1$ is a point of $X$ the same proof gives $\dist X  {p(z_1)}{p(z_2)} \leq \frac {3\pi \sinh r_0} {r_0} \distSC {z_1}{z_2}$. 
On the other hand, if $x$ belongs to a cone $C_i$, then Lemma \ref{projection distance cone} gives the same inequality.  
In the same way, we obtain 
\begin{displaymath}
	\dist X  {p(z_{m-1})}{p(z_m)} \leq \frac {3 \pi \sinh r_0} {r_0} \distSC {z_{m-1}}{z_m}
\end{displaymath}
Consequently, we have
\begin{eqnarray*}
	\dist X {p(x)}{p(x')} \leq \sum_{j=1}^{m-1} \dist X  {p(z_j)}{p(z_{j+1})} & \leq & \frac {3 \pi \sinh r_0} {r_0} \sum_{j=1}^{m-1}  \distSC {z_j}{z_{j+1}}\\
	&= & \frac {3 \pi \sinh r_0} {r_0}  l(C)\\
	&\leq &\frac {3 \pi \sinh r_0} {r_0} \left(\dist {\dot X} x {x'} + \epsilon \right)\\
\end{eqnarray*}
Hence $\dist X {p(x)}{p(x')} \leq \frac {3 \pi \sinh r_0} {r_0} \dist {\dot X} x {x'}$.
\end{proof}

\subsection{Uniform approximation of the distance on the cone-off}

In order to study the ultra limit of a sequence of cone-off spaces, we need to approximate the distance between two points of $\dot X$ by a chain such that the number of points involved in the chain only depends on the error and not on the base space $X$.
This point was already noted by M. Gromov in \cite{Gro01}.
More precisely, in this section we prove the following result:

\begin{prop}
\label{uniform approximation distance cone off}
	Let $ A\geq 1$. 
	There exists a constant $M$, depending only on $A$ and not on $r_0$, with the following property. 
	Let $\epsilon \in ]0, 1[$,  $X$  be a metric space and $Y=(Y_i)_{i \in I}$ a collection of subsets of  $X$. 
	Let $x$ and $x'$ be two points of the cone-off $\dot X(Y)$ such that $\dist{\dot X} x {x'} \leq A$.
	There exists a chain $C$ between $x$ and $x'$ with less than $\frac M{\sqrt \epsilon}$ points and such that $l(C) \leq \dist{\dot X} x {x'} + \epsilon$.
\end{prop}

\begin{proof}
	Let $\epsilon \in ]0,1[$. 
	Let  $x$ and $x'$ be two points of  $\dot X$ such that $\dist{\dot X} x {x'} \leq A$. 
	Using Lemma \ref{chain in X}, there is a chain $C = (z_1, \dots, z_n)$ between $x$ and $x'$ such that $l(C) \leq \dist{\dot X} x {x'} + \frac 12 \epsilon$ and for all $j \in \intval 2 {n-1}$, $z_j \in X$.
	We fix $\eta \in ]0,1[$, and construct a subchain of $C$ between $x$ and $x'$, denoted by $C_\eta = (z_{j_1}, \dots, z_{j_m})$, as follows.
	\begin{enumerate}
		\item Put $j_1 = 1$ and $j_2 = 2$.
		\item Let $k \geq 2$. We construct $j_{k+1}$ from $j_k$.
		\begin{itemize}
			\item If $j_k < n-1$ and $\dist X {z_{j_k}}{z_{j_k +1}} > \eta$, then $j_{k+1} = j_k +1$.
			\item If $j_k < n-1$ and $\dist X {z_{j_k}}{z_{j_k +1}} \leq \eta$, then $j_{k+1}$ is the largest  $j \in \{j_k+1, \dots, n-1\}$ such that $\dist X {z_{j_k}}{z_j} \leq \eta$
			\item If $j_k = n-1$, then $j_{k+1}=n$ and the process stops.
		\end{itemize}
	\end{enumerate}
	
	This construction removes from $C$ the small parts of the chain which may be contained in a cone. We prove now that it hardly changes the length of the chain.
	
\begin{lemm}[Comparison between the two chains]
	\label{comparison between chains}
	The chains $C_\eta$ and $C$ satisfy the following inequality:
	\begin{math}
		l(C_\eta) \leq l(C) + m \eta^3
	\end{math}, where $m$ is the number of points in $C_\eta$.
\end{lemm}

\begin{proof}
	We consider an integer $k \in  \{ 1 , \dots , m-2\}$.  There are two cases
\paragraph{First case: }
	Assume that $\dist X {z_{j_k}}{z_{j_{k+1}}} \leq \eta$. 
	The function $\mu$ given by Proposition \ref{function mu} has the following property:
	\begin{math}
		\forall t \in [0, \pi \sinh r_0], \ \mu(t) \geq t - t^3
	\end{math}. 
	Thus using the subadditivity of $\mu$ we obtain
	\begin{eqnarray*}
		\sum_{j = j_k}^{j_{k+1} - 1} \distSC{z_j}{z_{j+1}} & \geq & \sum_{j = j_k}^{j_{k+1} - 1} \mu(\dist X {z_j}{z_{j+1}}) \\
		& \geq & \mu (\dist X {z_{j_k}}{z_{j_{k+1}}}) \\
		& \geq & \dist X {z_{j_k}}{z_{j_{k+1}}} - \dist X {z_{j_k}}{z_{j_{k+1}}}^3\\
	\end{eqnarray*}
	Thus we have 
	\begin{math}
		\sum_{j = j_k}^{j_{k+1} - 1} \distSC{z_j}{z_{j+1}}  \geq \distSC {z_{j_k}}{z_{j_{k+1}}} - \eta^3
	\end{math}
\paragraph{Second case:} Assume that $\dist X {z_{j_k}}{z_{j_{k+1}}} > \eta$. By construction, we have $j_{k+1} = j_k +1$. Hence the last inequality remains true.

\paragraph*{}After summing over $k$ these inequalities, we finally obtain $l(C) \geq l(C_\eta) - m\eta^3$.
\end{proof}

\begin{lemm}[Estimation of $m$]
\label{estimation m}
	If $\eta \leq \frac 14$, then $m$, the number of points in the chain $C_\eta$, is less than $100 \frac A \eta$.
\end{lemm}

\begin{proof}
	Let $k \in \{2, \dots, m-3\}$.
	The two inequalities $\dist X {z_{j_k}}{z_{j_{k+1}}} \leq \frac 12 \eta$ and $\dist X {z_{j_{k+1}}}{z_{j_{k+2}}} \leq \frac 12 \eta$ cannot be both true. Indeed, if it was the case, $j_{k+1}$ will not be the largest  $j \in \{j_k+1, \dots, n-1\}$ such that $\dist X {z_{j_k}}{z_j} \leq \eta$.
	 Assume that $\dist X {z_{j_k}}{z_{j_{k+1}}} > \frac 12 \eta$ 
	 (the other case is symmetric).
	 Using the same estimation of $\mu$ as the one in the previous lemma, we obtain
	 \begin{displaymath}
			\distSC {z_{j_k}}{z_{j_{k+1}}} 
			\geq \mu(\dist X {z_{j_k}}{z_{j_{k+1}}}) 
			\geq \mu\left( \frac 12 \eta \right)
			\geq \frac 12 \eta - \frac 18 \eta^3
	 \end{displaymath}
	 Thus $\distSC {z_{j_k}}{z_{j_{k+1}}} + \distSC {z_{j_{k+1}}}{z_{j_{k+2}}}  \geq  \frac 12 \eta - \frac 18 \eta^3$.
	 After summing over $k$, the previous lemma gives
	 \begin{displaymath}
	 	\left\lfloor \frac {m-4}2 \right\rfloor \left(  \frac 12 \eta - \frac 18 \eta^3 \right)
	 	\leq l(C_\eta)
	 	\leq l(C) + m \eta^3
	 	\leq \dist{\dot X} x {x'} + \frac 12 \epsilon + m\eta^3
	 \end{displaymath}
	 We have the following inequality
	 \begin{displaymath}
			m\left(4 -17 \eta^2\right) \leq 50 \frac A \eta 
	 \end{displaymath}
	 Hence if $\eta \leq \frac 14$, then $4-17 \eta^2 \geq \frac 12$. 
	 It follows that $m$ must be bounded by $100 \frac A \eta$.
\end{proof}

\paragraph{End of the proof of Proposition \ref{uniform approximation distance cone off}}
Combining the two previous lemmas, we obtain 
\begin{displaymath}
	l(C_\eta) \leq l(C) +m \eta^3 \leq \dist{\dot X} x {x'} + \frac 12 \epsilon + 100A \eta ^2.
\end{displaymath}
If we chose $\eta = \frac 1 {10}\sqrt{\frac \epsilon{2 A}}$, then we have $l(C_\eta) \leq \dist{\dot X} x {x'} + \epsilon$. Moreover the number $m$ of points of $C_\eta$ is less than $1000A \sqrt{\frac{2A}{\epsilon}}$.
\end{proof}

\subsection{Contracting balls of the cone-off}
\label{section contracting balls of the cone-off}
In this section $X$ is a proper, geodesic, $\delta$-hyperbolic space.
We consider a collection $Y=(Y_i)_{i \in I}$ of closed strongly quasi-convex subsets of $X$.
We assume that $X$ satisfies the condition H(l), i.e.

\paragraph{Condition H(l) :}For all $x \in X$ and $r \in \R^+$, there exists a homotopy $h : \bar B (x, r) \times [0,1] \rightarrow X$ which contracts the closed ball $\bar B(x,r)$ to $\{x\}$ such that for all $x' \in \bar B(x,r)$, for all $t \in [0,1]$, $\distX x{h(x',t)} \leq \distX x{x'} +l$.

\paragraph{}We also assume that the $Y_i$'s satisfy the same condition.

\begin{prop}
\label{contracting ball in the cone-off}
	Let $x$ be a point of $\dot X$ and $r \in \R^+$.
	We assume that for all $i \in I$, $\distX x{v_i} > r$.
	Then the closed ball $\bar B(x,r)$ is contractible in $B(x, r+3l)$.
\end{prop}

\begin{proof}
	If there exists $i \in I$, such that $\bar B(x,r)$ is contained in a cone $C(Y_i)$, then we apply Proposition \ref{contracting ball contained in a cone}.
	Otherwise we proceed as follows.
	Let $i \in I$.
	Assume that $\bar B(x,r) \cap C(Y_i)  \neq \emptyset$.
	By Proposition \ref{contracting ball to the base in the cone}, there exists a homotopy $H_i : \bar B(x,r) \cap C(Y_i) \times [0,1] \rightarrow B(x, r+2l)$ satisfying the following properties.
	\begin{enumerate}
		\item $H_i$ contracts $\bar B(x,r) \cap C(Y_i)$ to a subset of $Y_i$.
		\item For all $x' \in \bar B(x,r)  \cap Y_i$, for all $t \in [0,1]$ $H_i(x',t)=x'$.
	\end{enumerate}
	Thus we may define a map $H : \bar B(x,r) \times [0,1] \rightarrow B(x,r+2l)$ such that its restriction to $\bar B(x,r) \cap C(Y_i) \times [0,1]$ is $H_i$, and for all $x' \in \bar B(x,r) \cap X$, for all $t \in [0,1]$, $H(x',t)=x'$.
	It follows that $H$ contracts $\bar B(x,r)$ to a subset of $B(x,r+2l)\cap X$.
	By condition H(l), this set is contractible in $B(x,r+3l)$.
\end{proof}

\subsection{Hyperbolicity of the cone-off  over an $\R$-tree}

\begin{prop}
\label{hyperbolicity cone-off arbitrary tree}
	Let $X$ be an $\R$-tree and $Y=(Y_i)_{i\in I}$  a collection of subtrees of $X$ such that two distinct elements of $Y$ share no more than one point.
	The cone-off $\dot X(Y)$ is $\ln 3$-hyperbolic.
\end{prop}
\rem In fact $\dot X(Y)$ is a CAT(-1) space, but we shall not use this point.

\paragraph{}This result is a consequence of the more particular case where $X$ is a finite simplicial tree.
\begin{lemm}
\label{hyperbolicity cone-off finite tree}
	Consider a \emph{finite simplicial} tree $X$ and a \emph{finite} collection $Y=(Y_i)_{i\in I}$ of subtrees of $X$ such that  two distinct elements of $Y$ share no more than one point. 
	
	The cone-off $\dot X(Y)$ is CAT(-1). In particular it is $\ln 3$-hyperbolic. 
\end{lemm}

\begin{proof}
	Each $Y_i$ is a tree, thus all the cones $C(Y_i)$ are CAT(-1) (cf. Lemma \ref{cone over R tree}). Consequently the cone-off $\dot X$ is obtained by gluing a finite number of CAT(-1)-spaces that share no more than one point. These spaces are the cones and the remaining parts of $X$ on which no cone is glued. It follows that $\dot X$ is CAT(-1) (cf \cite[Chap II.11 Th. 11.1]{BriHae99}).
\end{proof}

\begin{proof}[Proof of Proposition \ref{hyperbolicity cone-off arbitrary tree}]
	Let $x$, $y$, $z$ and $t$ be four points of $\dot X$. For all $n \in \N$, we can find a finite simplicial subtree $X_n$ of $X$ and a finite collection $Y^n$ of subtrees of $X_n$ such that 
	\begin{itemize}
	\item two distinct elements of $Y^n$ share no more than one point,
	\item $x,y,z,t$ belong to $\dot X_n(Y^n)$,
	\item for all $u,v \in \{ x,y,z,t\}$, we have $\displaystyle \lim_{n \rightarrow + \infty} \dist{\dot X_n(Y_n)} uv  = \dist {\dot X(Y)} uv$.
	\end{itemize}
	This can be done in the following way.
	Let $n \in \N$.
	For all pair of points in $\{x,y,z,t\}$ we consider a chain which approximates the distance between them, with an error smaller than $\frac 1n$.
	These chains contain a finite number of points.
	Thus we choose $X_n$ and $Y^n$ such that the chains have the same length in $\dot X(Y)$ and $\dot X_n(Y^n)$. 
	Since $\dot X_n$ is $\ln3$-hyperbolic (see Lemma \ref{hyperbolicity cone-off finite tree}), $x,y,z,t$ satisfy in these spaces the hyperbolicity condition. After taking the limit, we obtain in $\dot X$ the inequality
	\begin{math}
		\gro x z t \geq \min \left\{ \gro x y t , \gro y z t \right\} - \ln3
	\end{math}.
\end{proof}

\subsection{Hyperbolicity of the cone-off over a hyperbolic space}

In this section, we generalize the previous proposition when the base $X$ is a hyperbolic space. Let $\delta$ be a positive number. We consider a geodesic, $\delta$-hyperbolic space $X$ and a collection $Y=(Y_i)_{i \in I}$ of  closed $10 \delta$-quasi-convex subsets of $X$. In order to estimate the hyperbolicity of $\dot X(Y)$, we define a constant which controls how much two elements of $Y$ overlap.

\begin{defi}
	The largest piece of $Y$, denoted by $\Delta(Y)$ is the quantity
	\begin{displaymath}
		\Delta(Y)  = \sup_{i \neq j} \diam \left( Y_i^{+ 20 \delta} \cap Y_j^{+ 20 \delta} \right)
	\end{displaymath}
\end{defi}

Assume that $X$ is an $\R$-tree, so that $\delta=0$. Then $\Delta(Y)$ is zero if and only if two distinct elements of $Y$ share no more than one point.

\begin{theo}[Second hyperbolicity theorem]
\label{second hyperbolicity theorem}
	Let $\epsilon >0$. There exist $\delta, \Delta >0$ satisfying the following properties. Consider a geodesic, $\delta$-hyperbolic space $X$ and a collection $Y=(Y_i)_{i \in I}$ of closed, $10 \delta$-quasi-convex subsets of $X$, such that $\Delta(Y) \leq \Delta$. If $x_0$ is a point of the cone-off $\dot X(Y)$ whose distance to any vertex is greater than $\frac {r_0}2$, then the ball $B\left(x_0, \frac {r_0}9\right)$ is $(\ln 3 + \epsilon)$-hyperbolic.
\end{theo}

\rem This theorem is an extension of Proposition \ref{hyperbolicity cone-off arbitrary tree} for spaces that may be viewed as $\delta$-perturbed $\R$-trees. 
Thus it is possible to prove that $\dot X(Y)$ satisfies the CAT(-1)-condition with a small error, that depends only on $\delta$ and $\Delta$.
M. Gromov introduced in \cite{Gro01} the notion of ${\text{CAT(-1,}\epsilon\text{)}}$-spaces that formalizes this idea. 
It was also developed in \cite{DelGro08}.
Since we are only interested in the hyperbolicity of $\dot X(Y)$, we will not use it.

\paragraph{} In \cite[Hyperbolic coning lemma]{Gro01}, M. Gromov gave a quantitative statement of this result.
We propose here a detailed proof, that provides a qualitative version of the theorem.
The strategy is the following. 
Assuming that this theorem is false gives a family $X_n$ of $\delta_n$-hyperbolic counter-examples with $\delta_n$ tending to zero. 
Taking the limit gives the cone off over an $\R$-tree which we already know to be $\ln 3$-hyperbolic. 
This is a contradiction according to Corollary \ref{reverse ultra limit of hyperbolic space}. 
The point is to construct a local isometry between the cone-off over the ultra-limit of $(X_n)$ and the ultra-limit of the cones-off over $X_n$.

\begin{proof}
	Assume that the theorem is false. Then for all $n \in \N$, we can find 
	\begin{enumerate}
		\item a geodesic, $\delta_n$-hyperbolic space $X_n$, with $\delta_n = o(1)$,
		\item a collection $Y_n = (Y_{n,i_n})_{i_n \in I_n}$ of closed, $10 \delta_n$-quasi-convex subsets of $X_n$, with $\Delta(Y_n) = o(1)$,
		\item a point $x^0_n \in \dot X_n(Y_n)$ whose distance to any vertex is greater than $\frac {r_0}2$  and such that the ball $B\left(x^0_n, \frac{r_0}8\right)$ is not $(\ln3 + \epsilon)$-hyperbolic.
	\end{enumerate}
	
	We fix a non-principal ultra-filter $\omega$ in order to study the limit space $\limo \left(\dot X_n, x^0_n\right)$. 
	First we consider several objects. 
	\begin{itemize}
		\item $x^0 = \limo x^0_n$
		\item $X = \limo \left(X_n, p_n\left(x^0_n\right) \right)$ (Recall that $p_n$ is the projection from $\dot X_n$ onto $X_n$.)
		\item $I = \Pi_{n \in \N} I_n/\sim $ where $\sim$ is the equivalence relation defined by $i\sim j$ if $i_n = j_n$, \oasly.
		\item If $i$ is a sequence of $\Pi_{n \in \N} I_n$, we define $Y_i=\limo Y_{n,i_n}$.
	\end{itemize}
	
\begin{lemm}
	Let $i=(i_n)$ and $j=(j_n)$ be two sequences of $ \Pi_{n \in \N} I_n$. If $i_n = j_n$, \oasly, then $Y_i = Y_j$ otherwise, $\diam \left(Y_i \cap Y_j\right)=0$.
\end{lemm}	

\begin{proof}
	If  $i_n = j_n$, \oasly, the equality $Y_i = Y_j$ follows from the definition of the $\omega$-limit of a sequence of subsets.  In the other case, we have $i_n \neq j_n$, \oasly. Hence $\diam \left( Y_{n,i_n}^{+ 20 \delta_n} \cap Y_{n,j_n}^{+ 20 \delta_n}\right) \leq \Delta_n$, \oasly.
	Thus Corollary \ref{limit diameter intersection} gives that  $\diam \left(Y_i \cap Y_j\right) \leq \limo \diam \left( Y_{n,i_n}^{+ 20 \delta_n} \cap Y_{n,j_n}^{+ 20 \delta_n}\right) = 0$. 
\end{proof}
	
Thanks to the previous lemma, we may consider the collection $Y= (Y_i)_{i\in I}$.

\begin{lemm}
	The cone-off $\dot X(Y)$ is $\ln 3$-hyperbolic.
\end{lemm}

\begin{proof}
	Since for all $n \in \N$, $X_n$ is geodesic and $\delta_n$-hyperbolic with $\delta_n = o(1)$, $X$ is an $\R$-tree. 
	Furthermore, any $Y_{n,i_n}$ is a $10 \delta_n$-quasi-convex subset of $X_n$. 
	Thanks to the previous lemma, $Y$ is a collection of subtrees such that two distinct elements of $Y$ share no more than one point. 
	Applying Proposition \ref{hyperbolicity cone-off arbitrary tree}, $\dot X(Y)$ is $\ln 3$-hyperbolic.
\end{proof}

The next step is to produce a local isometry between $\dot X(Y)$, the cone-off over $\limo X_n$, and $\limo \left(\dot X_n, x^0_n\right)$.
For this purpose we consider the following maps.
\begin{displaymath}
\begin{array}{rccccrccc}
	\psi :	& X			& \rightarrow	& \limo \dot X_n	& \quad	& \psi_{i} : 	& C(Y_{i})						& \rightarrow	& \limo \dot X_n \\
			&	\limo x_n	& \rightarrow	& \limo x_n			&			&					& \left(\limo y_n,r\right)	&	\rightarrow	& \limo (y_n,r)		\\
\end{array}
\end{displaymath}

These maps induce an function $\dot \psi$ from $\dot X$ to $\limo \dot X_n$, such that its restriction to $X$ (\resp $C(Y_{i})$) is $\psi$ (\resp $\psi_{i}$). At first, we prove that this map is 1-lipschitz, then we prove that it induces a local isometry.

\begin{lemm}
	Let $x$ and $x'$ be two points of $\dot X$. We have $\distSC x {x'} \geq \distX {\dot \psi (x)}{\dot \psi(x')}$.
\end{lemm}

\begin{proof}
	We distinguish three cases.
	\begin{enumerate}
		\item Assume that there is $i \in I$ such that  $x, x' \in C(Y_{i})$. Then we can write $x = (y,r)$ and $x'=(y',r')$ where $y = \limo y_n$ and $y' = \limo y'_n$ are two points of $Y_{i}$. In this situation we have
		\begin{eqnarray*}
			\distSC x {x'} & = & \dist{C(Y_{i})} x {x'} \\
			& = & \arccosh\left( \cosh r \cosh r' - \sinh r \sinh r' \cos \angle y {y'} \right)\\
		\end{eqnarray*}
		By continuity, it gives.
		\begin{eqnarray*}
			\distSC x {x'}	& =		& \limo \arccosh\left( \cosh r \cosh r' - \sinh r \sinh r' \cos \angle {y_n} {y'_n} \right) \\
									& =		& \limo \dist{C(Y_{n,i_n})} {(y_n,r)}{(y'_n,r')} \\
									& \geq	& \limo \dist{\dot X_n} {(y_n,r)}{(y'_n,r')} \\
									& = 		&  \distX {\dot \psi (x)}{\dot \psi(x')}
		\end{eqnarray*}
		\item Assume $x=\limo x_n$ and $x' = \limo x'_n$ belong to $X$, but there is no $i \in I$ such that $x,x' \in Y_{i}$.
		In this case $\distSC x {x'} = \dist X x {x'} = \limo \dist{X_n} {x_n}{x'_n}$. 
		However, for all $n \in \N$, we have $\dist{X_n} {x_n}{x'_n} \geq \dist{\dot X_n} {x_n}{x'_n}$. 
		Thus $\distSC x {x'} \geq \limo \dist{\dot X_n} {x_n}{x'_n} = \distX {\dot \psi (x)}{\dot \psi(x')}$.
		\item In all other cases, $\distSC x {x'} = + \infty$. There is nothing to prove.
	\end{enumerate}
\end{proof}

\begin{coro}
	The map $\dot \psi : \dot X \rightarrow \limo \dot X_n$ is 1-Lipschitz, where $\dot X$ is the cone-off over $\limo X_n$.
\end{coro}
	
\begin{proof}
	Let $x$ and $x'$ be two points of $\dot X$. Consider a chain $C=(z_1,\dots z_m)$ of points of $\dot X$ between $x$ and $x'$. Using the previous lemma, we have
	\begin{displaymath}
		\distX {\dot \psi (x)}{\dot \psi (x')} 
		\leq \sum_{j=1}^{m-1} \distX {\dot \psi (z_j)}{\dot \psi (z_{j+1})} 
		\leq \sum_{j=1}^{m-1} \distSC {\dot \psi (z_j)}{\dot \psi (z_{j+1})} 
		= l(C)
	\end{displaymath}
	After taking the infimum over all chains between $x$ and $x'$, we obtain $\distX {\dot \psi (x)}{\dot \psi (x')}  \leq \dist{\dot X} x {x'}$.
\end{proof}

We now construct a partial inverse function of $\dot \psi$.
\begin{lemm}
	There is a map $\dot \phi : B\left(x^0, \frac {r_0}3\right) \subset \limo \dot X_n \rightarrow \dot X$ such that 
	$\dot \phi$ induces a bijection onto the ball $B\left(\dot \phi(x^0), \frac {r_0}3\right)$, whose inverse is $\dot \psi$.
\end{lemm}

\begin{proof}
	Let $x = \limo x_n$ be a point of $B\left(x^0, \frac {r_0}3\right)$. 
	By construction, the distance between $x^0_n$ and any vertex of $\dot X_n$ is greater than $\frac {r_0}2$. 
	Thus applying Lemma \ref{comparison distance base cone-off}, we have 
	\begin{displaymath}
	\dist {X_n} {p_n(x^0_n)}{p_n(x_n)} \leq \frac{3 \pi \sinh r_0}{r_0} \dist {\dot X_n} {x^0_n}{x_n}
	\end{displaymath}
	It follows that $\left( \dist {X_n} {p_n(x^0_n)}{p_n(x_n)} \right)$ is \oeb. Hence $\limo p_n(x_n)$  defines  a point  in $X$. We now distinguish two cases.
	\begin{enumerate}
		\item If there is $i \in I$ such that $x_n$ belongs to $C(Y_{n,i_n})$ \oasly, then $x_n$ can be written $x_n = (p_n(x_n), r_n)$ \oasly. 
		Since $(r_n)$ is bounded, we may consider $r = \limo r_n$. 
		We define $\dot \phi (x)$ as the point  $(\limo p_n(x_n), r)$  of $C(Y_{i})$.
		\item If $x_n$ belongs to $X_n$ \oasly, then we define  $\dot \phi (x)$ as the point  $\limo p_n(x_n)$  of $X$.
	\end{enumerate}
	
	The properties of $\dot \phi$ are satisfied.
\end{proof}

\begin{lemm}
	Let $x = \limo x_n$ and $x' = \limo x'_n$ be two points of $B\left(x^0, \frac {r_0}3\right)$ such that $\left(\distSC {x_n}{x'_n}\right)$ is  \oeb.
	We have $\limo \distSC {x_n}{x'_n} \geq \distSC {\dot \phi (x)}{\dot \phi (x')}$.
\end{lemm}

\begin{proof}
 We distinguish two cases.
 \begin{enumerate}
 	\item If there is $i\in I$ such that $x_n$ and $x'_n$ belong to $C(Y_{n,i_n})$ \oasly, then $\distSC {x_n}{x'_n} = \dist {C(Y_{n,i_n})} {x_n}{x'_n}$ \oasly. By continuity we have
 	\begin{displaymath}
		\limo \dist {C(Y_{n,i_n})} {x_n}{x'_n} = \dist{C(Y_{i})} {\dot \phi (x)}{\dot \phi (x')} \geq \distSC  {\dot \phi (x)}{\dot \phi (x')}
 	\end{displaymath}
 	Thus $\limo \distSC {x_n}{x'_n} \geq \distSC  {\dot \phi (x)}{\dot \phi (x')}$.
 	
 	\item If $x_n, x'_n \in X_n $ \oasly, but there is no $i \in I$ such that $x_n$ and $x'_n$ belongs to $C(Y_{n,i_n})$ \oasly, then $\distSC {x_n}{x'_n} = \dist {X_n} {x_n}{x'_n}$ \oasly. In this case we have
 	\begin{displaymath}
		\limo \dist {X_n}{x_n}{x'_n}  = \dist X  {\dot \phi (x)}{\dot \phi (x')} \geq \distSC  {\dot \phi (x)}{\dot \phi (x')}
 	\end{displaymath}
 	Thus $\limo \distSC {x_n}{x'_n} \geq \distSC  {\dot \phi (x)}{\dot \phi (x')}$.
 \end{enumerate}
\end{proof}

\paragraph{}
The proof of the next corollary uses the uniform approximation of the distance on the cone-off. 
Indeed, if $x = \limo x_n$ and $x' = \limo x'_n$ are two points of $B\left(x^0, \frac{r_0}9\right)$, we can find for each $n$ a chain $C_n$ of $\dot X_n$ that approximates $\distX {x_n}{x'_n}$ with a given error. 
However it is difficult to give a sense to the $\omega$-limit of $C_n$ if the number of points of $C_n$ is not uniformly bounded.
That is why we previously proved Proposition \ref{uniform approximation distance cone off}.

\begin{coro}
\label{local isometry cone off}
	The restriction of  $\dot \phi$ to the ball $B\left(x^0, \frac{r_0}9\right)$ is 1-Lipschitz
\end{coro}

\begin{proof}
	Consider two points $x= \limo x_n$ and $x'= \limo x'_n$ of $B\left(x^0, \frac{r_0}9\right)$.
	Let $\epsilon >0$ such that $\distX x{x'} + \epsilon < \frac {r_0}9$. 
	By Proposition \ref{uniform approximation distance cone off}, there is a number $m$ depending only of $r_0$ such that for all $n \in \N$, there is a chain $C_n = (z^1_n, \dots, z^m_n)$ between $x_n$ and $x'_n$ with $l(C_n) \leq \dist {\dot X_n} {x_n}{x'_n} + \epsilon  < \frac {r_0}9 $.
	Notice that for all $1\leq j \leq m$, $\dist {\dot X_n}{x^0_n}{z^j_n} \leq \dist {\dot X_n}{x^0_n}{x_n} + l(C_n) < \frac {r_0}3$. 
	Thus we can consider the points $z^j = \limo z^j_n$.
	The previous lemma gives
	\begin{eqnarray*}
		\dist{\dot X}{\dot \phi(x)}{\dot \phi(x')} 
		& \leq & \sum_{j=1}^{m-1} \distSC {\dot \phi(z^j)}{\dot \phi(z^{j+1})}\\
		& \leq & \limo l(C_n) 
		\leq \limo \dist {\dot X_n}{x_n}{x'_n} + \epsilon
	\end{eqnarray*}
	Hence for all $\epsilon >0$, we have $\dist{\dot X}{\dot \phi(x)}{\dot \phi(x')}  \leq \distX x {x'} +\epsilon$. 
	Thus $\dot \phi$ is 1-Lipschitz.
\end{proof}

\begin{coro}
	The map $\dot \phi$ induces an isometry from the ball $B\left(x^0, \frac{r_0}9\right)$ onto its image.
\end{coro}

\begin{proof}
	We already know that $\dot \phi$ is a 1-Lipschitz bijection from $B\left(x^0, \frac{r_0}9\right)$ onto its image. However its inverse function $\dot \psi$ is also 1-Lipschitz. Hence $\dot \phi$ preserves the distances.
\end{proof}

\paragraph{End of the proof of the theorem}
We have just proved that $B\left(x^0, \frac{r_0}9\right)$ is isometric to a subset of $\dot X(Y,r_0)$. 
Since $\dot X(Y,r_0)$ is $\ln 3$-hyperbolic, so is $B\left(x^0, \frac{r_0}9\right) = \limo B\left(x^0_n,\frac{r_0}9\right)$. 
Consequently there exists $n \in \N$ such that $B\left(x^0_n,\frac{r_0}9\right)$ is $(\ln3+ \epsilon)$-hyperbolic. 
Contradiction.
	
\end{proof}

\subsection{Length structure on the cone-off}

	\paragraph{} In order to apply the Cartan-Hadamard theorem, we need a length structure on the cone-off. 
	But the metric $|\ . \ |_{\dot X}$ is not necessary a length metric.
	We study here the difference between $|\ . \ |_{\dot X}$ and the length metric $d_{\dot X}$ induced by $|\ . \ |_{\dot X}$.
	We will see that $d_{\dot X}$ hardly changes the geometry of $\dot X$.
	Thus if $(\dot X, |\ . \ |_{\dot X})$ is hyperbolic, then so is ($\dot X,d_{\dot X})$.
	
	\paragraph{} From now on, $X$ is a geodesic, $\delta$-hyperbolic space, and $Y=(Y_i)_{i \in I}$ a collection of strongly quasi-convex subsets of $X$. 
	We recall that a strongly quasi-convex set $Y_i$ satisfies the following property: 
	for all $x, x' \in Y_i$ there exist $y, y' \in Y_i$ such that the path $\geo{x}{y} \cup \geo{y}{y'} \cup \geo{y'}{x'} \subset Y_i$ and $\distX{x}{y}, \distX{x'}{y'} \leq 10 \delta$. 
	In particular this condition is satisfied if $Y_i$ is a cylinder (see Proposition \ref{cylinder quasi-convex}). 
	$\dot X(Y)$ is the cone-off constructed as in Subsection \ref{part definition cone-off}. 
	
	\begin{lemm}
		Let $i \in I$. For all $x, x' \in C(Y_i)$, we have $\distLbis{\dot X}{x}{x'} \leq \dist{C(Y_i)}{x}{x'} + 40 \delta$.
	\end{lemm}
	
	\begin{proof}
		We denote by $x=(y,r)$ and $x'=(y',r')$ two points of the cone $C(Y_i)$. We have assumed that $Y_i$ was strongly quasi-convex, thus we can find two points $z$ and $z'$ in $Y_i$ such that the geodesics $\geo {y}{z}$, $\geo{z}{z'}$ and $\geo{z'}{y'}$ are contained in $Y_i$ and $\dist X{y}{z}, \dist X{y'}{z'}\leq 10 \delta$. 
		
		By Proposition \ref{geodesic cone BH}, we can find a geodesic $c_0$ (\resp $c$, $c'$) between $(z,r)$ and $(z',r')$ (\resp between $(y,r)$ and $(z,r)$,  between $(y',r')$ and $(z',r')$). 
		Since $\dist X{y}{z} ,\dist X{y'}{z'}\leq 10 \delta$, we have 
		\begin{displaymath}
		\dist{C(Y_i)}{(y,r)}{(z,r)}, \dist{C(Y_i)}{(y',r')}{(z',r')}  \leq 10 \delta
		\end{displaymath}
		It follows that $\dist{C(Y_i)}{(z,r)}{(z',r')}\leq \dist{C(Y_i)}{x}{x'} +20 \delta$. 
		Thus by composing the geodesics $c$, $c_0$ and $c'$, we obtain a path from $x$ to $x'$ whose length is no more than $\dist{C(Y_i)}{x}{x'} + 40 \delta$.
	\end{proof}
	
	\begin{coro}
	\label{comparison length metric norm}
		For all $x,x' \in \dot{X}$, we have $\distL x{x'} \leq \distSC x {x'}+ 40 \delta$.
	\end{coro}
	
	\begin{proof}
	Let $x$ and $x'$ be two points of $X$. We distinguish three cases.
	\begin{itemize}
		\item If there exists $i \in I$ such that $x,x' \in C(Y_i)$, then $\distSC x {x'} = \dist{C(Y_i)}x{x'}$. Thus the inequality is given by the previous lemma.
		\item If $x,x' \in X$, but there is no  $i \in I$ such that $x,x' \in C(Y_i)$, then $\distSC x {x'} = \dist Xx{x'}$. There is a geodesic of $X$ between $x$ and $x'$. It gives a path in $\dot X$, whose length is no more than $\dist Xx{x'}$. It follows that $\distL xy \leq \distSC x {x'}$.
		\item In all the other cases, $\distSC x {x'} = + \infty$. There is nothing to prove.
	\end{itemize}
	\end{proof}
	
\begin{prop}
	\label{comparison length metric metric}
	Let $A \geq 1$ and $\eta >0$.
	There exists a constant $\delta >0$ depending only on $A$ and $\eta$ with the following property.
	 Let $X$ be a geodesic, $\delta$-hyperbolic space, and $Y=(Y_i)_{i \in I}$ a collection of strongly quasi-convex subsets of $X$.
	 The identity map from $\left(\dot X,|\ .\ |\right)$ onto $\left(\dot X, d\right)$ induces a $(1,\eta)$-quasi-isometry on any ball of radius $A$.
\end{prop}

\begin{proof}
	Let $x$ and $x'$ be two points of $\dot X$ such that $\dist{\dot X}x{x'} \leq 2A$. Using Proposition \ref{uniform approximation distance cone off} there exists a constant $M(A)$, an integer $m\leq \frac{M(A)}{\sqrt \delta}$ and a chain $C=\left(z_1,\dots,z_m\right)$ between $x$ and $x'$ such that $l(C) \leq \dist{\dot X}x{x'} + \delta$. Thanks to Corollary \ref{comparison length metric norm} we have
	\begin{displaymath}
		\distL x{x'}  
		\leq \sum_{j=1}^{m-1} \distL {z_j}{z_{j+1}} 
		\leq  l(C) + 40m\delta
		\leq \dist{\dot X}x{x'} + \delta + 40M(A)\sqrt \delta
	\end{displaymath}
	Thus we have
	\begin{displaymath}
		\dist{\dot X}x{x'} \leq \distL x{x'} \leq \dist{\dot X}x{x'} + \delta + 40M(A)\sqrt \delta
	\end{displaymath}
	Consequently, if $\delta$ is small enough, the identity map from $\left(\dot X,|\ .\ |\right)$ onto $\left(\dot X, d\right)$ induces a $(1,\eta)$-quasi-isometry on any ball of radius $A$.
\end{proof}

\rem This last proposition combined with Proposition \ref{quasi-isometry hyperbolicity} tells us that if $\left(\dot X,|\ .\ |\right)$ is locally hyperbolic, then so is $\left(\dot X, d\right)$.

%% file: 4-small_cancellation.tex
\section{Small cancellation theory}
\label{part small cancellation}

\subsection{Orbifold}

In this subsection, we introduce vocabulary concerning orbifolds. For more details about these objects see \cite[Part III $\mathcal{G}$]{BriHae99}

\subsubsection*{Definition and length structure}

\begin{defi}[Rigidity]
The action of a group $G$ on a topological space $X$ is rigid if it satisfies the following property : for all $g \in G$ if there is an open $U \subset X$ such that $\restriction gU = \id_U$ then $g=1$.
\end{defi}

\begin{defi}[Orbifold]
	Let $Q$ be a topological space. We say that $Q$ is an orbifold if there exists a collection $\left( U_i, \varphi_i\right)_{i\in I}$, where $U_i$ is a topological space and $\varphi_i$ a continuous map from $U_i$ into $Q$ satisfying the following properties.
	
	\begin{enumerate}
		\item $Q = \bigcup_{i\in I} \varphi_i(U_i)$.
		\item For all $y \in \varphi_i(U_i)$, for all $x \in \varphi_i^{-1}\left(\{y\} \right)$, there exists a finite, rigid group of homeomorphisms of $U_i$, $G_x$, fixing $x$, such that for all $g \in G_x$, $\varphi_i \circ g = \varphi_i$ and such that the restriction of $\varphi_i$ to a neighbourhood of $x$, $V_x$, induces a homeomorphism from $V_x/G_x$ onto its image.
		\item For all $x_i \in U_i$ and $x_j \in U_j$, such that $\varphi_i(x_i) = \varphi_j(x_j)$, there exists  a homeomorphism $ \theta_{j,i}$ from a neighbourhood of $x_i$ onto a neighbourhood of $x_j$ such that $\varphi_i = \varphi_j \circ \theta_{j,i}$.
		\item For all $i \in I$, $\varphi_i$ lifts paths and homotopies, i.e. if $c: [0,1] \rightarrow Q$ (\resp $H:[0,1]\times[0,1] \rightarrow Q$) is a continuous path (\resp a homotopy), there exists $0=t_0<\dots <t_p=1$ a subdivision of $[0,1]$ (\resp  $0=t_0<\dots < t_p=1$ and $0=u_0 < \dots <u_q=1$  subdivisions of $[0,1]$) such that $\restriction c {[t_r,t_{r+1}]}$ (\resp $\restriction H {[t_r,t_{r+1}]\times[u_s,u_{s+1}]}$) lifts in one of the $U_i$.
	\end{enumerate}
\end{defi}	
	\voc $(U_i, \varphi_i)$ is called a chart of $Q$. The set of charts is an atlas. The map $\theta_{j,i}$ is a transition map, and the group $G_x$ is an isotropy group.
	
	\begin{defi}[Length structure]
		The orbifold defined as above is endowed with a length structure if 
		\begin{enumerate}
			\item the spaces $U_i$ are endowed with a length structure,
			\item for all $x \in U_i$, the isotropy group $G_x$ is an isometry group for the length structure in $U_i$,
			\item the transition maps $\theta_{j,i}$ are isometries, with respect with the length structures in $U_i$ and $U_j$.
		\end{enumerate}
	\end{defi}
	
	In this case, we can measure the length of a path, by measuring the length of its lift.
	
	\begin{defi}[$\sigma$-useful length structure]
		Let $\sigma$ be a positive number. The length structure defined as above is said to be $\sigma$-useful if for all  $y \in Q$, there exists a chart $(U_i, \varphi_i)$ and a point $x \in \varphi_i^{-1}\left(\{y\} \right)$, such that
		\begin{enumerate}
			\item the restriction $\varphi_i : B(x,\sigma) \rightarrow B(y,\sigma)$ is onto,
			\item this restriction lifts the paths starting in $y$, whose lengths are less than $\frac \sigma 2$,
			\item this restriction lifts the homotopies $H : [0,1]\times[0,1] \rightarrow Q$ satisfying $H(0,0)=y$, and for all $t_0 \in [0,1]$ (\resp $u_0 \in [0,1])$ the length of the path $u\rightarrow H(t_0,u)$ (\resp $t \rightarrow H(t,u_0)$) is shorter than $\frac \sigma 2$.
		\end{enumerate}
		$\left(U_i, \varphi_i,x\right)$ is called a $\sigma$-useful chart.
	\end{defi}
	
	\begin{defi}[$\sigma$-locally $\delta$-hyperbolic length structure]
		Let $\sigma >0$ and $\delta >0$. The $\sigma$-useful length structure, defined as above, is said to be $\sigma$-locally $\delta$-hyperbolic if for all $y \in Q$, there exists a $\sigma$-useful chart $\left(U_i, \varphi_i,x\right)$ such that the ball $B(x, \sigma)$ is $\delta$-hyperbolic.
	\end{defi}

\subsubsection*{Topology of orbifolds}

	If $Q$ is an orbifold, we can define the $\mathcal{G}$-paths and the homotopy of two $\mathcal{G}$-paths. (cf. \cite{BriHae99} or \cite{DelGro08}). This leads to the definition of the fundamental group of the orbifold $Q$ denoted by  $\pi^\text{orb}_1 (Q)$. We may also define the notion of covering and universal covering of $Q$ in the sense of orbifolds. (cf. \cite{BriHae99})
	
\ex
	
	Let  $X$ be a geodesic space and $G$ a group whose action on $X$ is rigid and proper. We denote by $Q$ the quotient $X/G$, and by $q: X \rightarrow Q$ the canonical projection. $Q$ may be endowed with an orbifold structure with one chart $(X,q)$. Indeed, for all $x \in X$, the isotropy group $G_x$ is necessarily finite. Moreover $q$ induces a local isometry from  $X/G_x$ onto its image. If $X$ is simply connected, the map $q:X \rightarrow Q$ is also the universal cover of $Q$ and $G= \pi^\text{orb}_1(Q)$. Such an orbifold is said to be developable.

\subsubsection*{Cartan Hadamard Theorem}

\begin{theo}[{\cite[Th. 4.3.1]{DelGro08}}]
\label{cartan-hadamard-orbifold}
	Let $\delta >0$ and $\sigma > 10^5\delta$. Consider an orbifold $Q$ with a $\sigma$-locally $\delta$-hyperbolic length structure.
	Then $Q$ is developable and its universal cover $X$ is $200 \delta$-hyperbolic.
	Let $(U,\varphi,x)$ be  a $\sigma$-useful chart. 
	If $z$ is a preimage in $X$ of the point $y=\varphi(x)$, then the developing map $(U,x) \rightarrow (X,z)$ induces an isometry from $B\left (x,\frac \sigma {10} \right)$ onto its image.
\end{theo}

\subsection{Statement of the very small cancellation theorem}
\label{Statement of the very small cancellation theorem}
	
	\nota If $G$ is a group acting on a space $X$, and $Y$ a subset of $X$, we denote by $\stab Y$ the subgroup of $G$ that preserves $Y$, i.e. 
	\begin{displaymath}
		\stab Y = \left\{ g \in G\ |\ gY = Y \right\}
	\end{displaymath}
	
	We define the notion of rotation family introduced by M. Gromov in \cite{Gro03}.
	\begin{defi}[Rotation family]
		Let $\left(H_i\right)_{i \in I}$ be a family of subgroups of $G$ and $(Y_i)_{i\in I}$ a collection of pairwise distinct subspaces of $X$. We say that $(Y_i, H_i)_{i\in I}$ is a rotation family if 
		\begin{enumerate}
			\item for all $i\in I$, $H_i$ is a finite index normal subgroup of $\stab {Y_i}$,
			\item there is an action of $G$ on $I$, which is compatible with the one on $X$, that is for all $g \in G$ and $i \in I$, we have $Y_{gi} = gY_i$ and $H_{gi} = gH_i g^{-1}$.
		\end{enumerate}
		
	\end{defi}
	
	\begin{theo}[Very small cancellation theorem]
	\label{very small cancellation}
		There exist two positive numbers $\delta_0$ and $\Delta_0$ satisfying the following property.
		
		Let  $X$ be a geodesic, simply connected, $\delta$-hyperbolic space and  $G$ a group acting properly, by isometries on $X$. Let $(Y_i,H_i)_{i \in I}$ be a rotation family, such that each $Y_i$ is strongly-quasi-convex.
		Let $ \rho = \min_{i \in I } \rinj{H_i}X$, $N$ be the normal subgroup of $G$ generated by the $H_i$'s and $\bar{G}$ the quotient group $G/N$. Assume also that 
		\begin{displaymath}
		\frac  \delta \rho \leq \delta_0 \ \text{ and } \ \frac{\Delta(Y)}\rho\leq \Delta_0
		\end{displaymath}
		
		Then, there exists a simply connected, hyperbolic, metric space $\bar X$ such that $\bar{G}$ acts properly by isometries on $\bar X$.
		
		Moreover if $G$ (\resp $H_i$) acts co-compactly on $X$ (\resp $Y_i$) and $I/G$ is finite, then $\bar X / \bar G$ is compact. In particular $\bar G$ is hyperbolic.
	\end{theo}
	
	\rem In this theorem, $\Delta(Y)$ and $\rho$ respectively play the role of the length of the largest piece and the length of the smallest relation in the usual small cancellation theory.
	
	It is important to notice that the constants $\delta_0$ and $\Delta_0$ are independent from the space $X$. This is useful in order to construct by iteration a sequence of hyperbolic groups, as it is done in \cite{Gro03}, \cite{DelGro08}  or \cite{ArzDel00}.

\subsection{Proof of the very small cancellation theorem}
\subsubsection*{Construction of an orbifold}

\paragraph{}First we have to fix several constants in order to apply the Cartan-Hadamard theorem (\ref{cartan-hadamard-orbifold}) and the two hyperbolicity theorems (\ref{first hyperbolicity theorem} and \ref{second hyperbolicity theorem}). 

\paragraph{} We consider a positive number $\epsilon$ and choose a radius $r_0$ such that $r_0 > 10^6\left( \ln3 + \epsilon\right)$. With such constants we can apply the Cartan-Hadamard theorem to a $\frac {r_0}{10}$-locally $(\ln3 + \epsilon)$-hyperbolic orbifold. By Proposition \ref{quasi-isometry hyperbolicity}, there exists a positive number $\eta$ with the following property. Consider two metric spaces $X$ and $X'$ and a $(1,\eta)$-quasi-isometry  $f:X \rightarrow X'$. If $X'$ is $ \left( \ln3 + \frac\epsilon 2\right)$-hyperbolic then $X$ is $\left( \ln3 + \epsilon\right)$-hyperbolic. From now on we will work with the rescaled metric space $X_\rho = \frac{2 \pi \sinh r_0}\rho X$. Thus for all $i \in I$, $\rinj{H_i}{X_\rho} \geq  2 \pi \sinh r_0$.

\paragraph{}We can find $\delta_0, \Delta_0 >0$ only depending on $r_0$ and $\epsilon$ such that if  $\frac  \delta \rho \leq \delta_0$ and $\frac{\Delta(Y)}\rho\leq \Delta_0$, then we have the following.

\begin{enumerate}
	\item Assume that $x_0$ is a point of $\dot X_\rho(Y)$, whose distance to a vertex is at least $\frac{r_0}2$.
	Then the ball $B\left(x_0, \frac{r_0}9\right)$ of  $\left(\dot X_\rho, |\ . \ |_{\dot X_\rho}\right)$ is  $ \left( \ln3 + \frac\epsilon 2\right)$-hyperbolic (see. Theorem \ref{second hyperbolicity theorem}).
	\item For all $i \in I$ the cone $C(Y_i)/H_i$ is $ \left( \ln3 + \frac\epsilon 2\right)$-hyperbolic. (see Theorem  \ref{first hyperbolicity theorem}).
	\item The identity map from $\left(\dot X_\rho, d\right)$ onto $\left(\dot X_\rho, |\ . \ |\right)$ restricted to any ball of radius $r_0$ is a $(1,\eta)$-quasi-isometry. (see Proposition \ref{comparison length metric metric}).
\end{enumerate}

Thus if $x_0$ is a point of $\dot X_\rho(Y)$, whose distance to a vertex is at least $\frac{r_0}2$, the ball $B\left(x_0, \frac{r_0}{10}\right)$ of $\left(\dot X_\rho, d\right)$ is $\left( \ln3 + \epsilon\right)$-hyperbolic and for all $i \in I$, the cones $C(Y_i)/H_i$ with the length metric are $\left( \ln3 + \epsilon\right)$-hyperbolic.

\begin{lemm}
	The action of $G$ on $X_\rho$ extends to an action by isometries of $G$ on $\dot X_\rho$.
\end{lemm}

\begin{proof}
	We define this action such that its restriction to $X_\rho$ is the action of $G$ on $X$. Let $i \in I$. Consider a point $x=(y,r)$ of $C(Y_i)$ and an element  $g$ of $G$. We define $g.x$ as the point $(gy,r)$ of $C(Y_{g.i})$. We check that for all $x,x' \in \dot X_\rho$ and for all $g \in G$ we have $\distSC {gx}{gx'} = \distSC x {x'}$. It follows that the action of $G$ preserves the distances $| \ . \ |_{\dot X_\rho}$ and $d$.
\end{proof}

From now on, we denote by $Q$ the quotient space $\dot X_\rho(Y)/G$ endowed with the quotient topology. The canonical projection $\dot X_\rho \rightarrow Q$ is denoted by $q$. Then we define two kind of charts.
\begin{itemize}
		\item The first one is $(U,q)$ where $U$ is the cone-off $\dot X_\rho(Y)$ from which we have removed the vertices.
		\item Let  $i \in I$, we define $U_i = \left(C(Y_i) \setminus \iota_i (Y_i)\right)/H_i$. The composition $C(Y_i) \rightarrow \dot X_\rho \rightarrow Q$ induces an map $q_i : U_i \rightarrow Q$. The second type of charts is $(U_i, q_i)$.
	\end{itemize}

\begin{lemm}
	The charts defined previously endow $Q$ with a structure of orbifold.
\end{lemm}

\begin{proof}
	The action of $G$ on $U$ is proper. Moreover the stabilizer of the vertex $v_i$ of the cone $U_i$ is exactly the finite group $\stab {Y_i} / H_i$.  We check that the atlas $(U,q)$, $(U_i,q_i)$ defines a structure of orbifold on $Q$.
\end{proof}

\subsubsection*{Properties of the orbifold $Q$}

\begin{lemm}
	The structure of orbifold on $Q$ defined as above is $\frac{r_0}{10}$-locally $(\ln3 + \epsilon)$-hyperbolic.
\end{lemm}
\begin{proof}
It is a consequence of the two hyperbolicity theorems: the constants $\delta_0$ and $\Delta_0$ have been chosen in such a way, that the structure of orbifold is $\frac{r_0}{10}$-locally $(\ln3 + \epsilon)$-hyperbolic.
\end{proof}

\begin{coro}
	The orbifold $Q$ is developable and its universal cover $\bar X$ is $\bar \delta$-hyperbolic, with $\bar \delta = 200(\ln3 + \epsilon)$.
\end{coro}

\begin{proof}
	This is an application of the Cartan-Hadamard theorem (see Theorem \ref{cartan-hadamard-orbifold}) to the orbifold $Q$ with his $\frac {r_0}{10}$-locally $(\ln3 + \epsilon)$-hyperbolic length structure.
\end{proof}

\begin{prop}
	The group $\bar G$ acts properly  by isometries on $\bar X$.
\end{prop}

\begin{proof}
	We prove that $\bar G = \pi^\text{orb}_1(Q)$. The charts $U$ and $U_i$ are simply connected. 
	It follows $\pi^\text{orb}_1\left(U/G\right) = G $ and $\pi^\text{orb}_1\left(U_i/\left(\stab {Y_i}/H_i\right)\right) = \stab {Y_i}/H_i$. Applying to $Q$ the Van Kampen's theorem for orbifolds, we obtain $\pi^\text{orb}_1(Q) = \stab {Y_i}/H_i *_{\stab {Y_i}} G  = \bar G$ (see \cite{HaeDu84}). Thus $\bar G$ acts properly on $\bar X$.
\end{proof}

\begin{prop}
	 If $G$ (\resp $H_i$) acts co-compactly on $X$ (\resp $Y_i$) and $I/G$ is finite, then $\bar X / \bar G$ is compact. In particular $\bar G$ is hyperbolic.
\end{prop}

\begin{proof}
	Since there are, up to a translation by an element of $G$, a finite number of $Y_i$, $Q$ is obtained by gluing a finite number of compact cones, on the compact space $X/G$. Thus $Q = \bar X/\bar G$ is compact. Hence the action of $\bar G$ on $\bar X$ is proper, co-compact.
\end{proof}

\cmath{Theorem \ref{asphericity universal cover} and its proof have been rewritten.}
\begin{theo}
	\label{asphericity universal cover}
	Assume that $X$ is a $n$-dimensional simplicial complex with $8(n+1) \bar \delta \leq \frac {r_0}{100}$.
	We suppose that for all $x \in X$, for all $r \in \R^+$, there exists a homotopy $h : \bar B(x, r) \times [0,1] \rightarrow X$ contracting $\bar B(x,r)$ to $\{x\}$ such that for all $x' \in \bar B(x,r)$, $\distX x{h(x',t)} \leq \distX x{x'} + \frac {\bar \delta \rho}{2 \pi \sinh r_0}$.
	We also suppose that the $Y_i$'s have the same property.
	Then $\bar X$ is contractible.
\end{theo}

\begin{proof}
Let $\bar x$ be a point of $\bar X$ and $r \in [0, 8(n+1) \bar \delta]$.
We denote by $\bar B$ the closed ball $\bar B(\bar x,r)$ of $\bar X$.
We distinguish two cases.
	\subparagraph{Case 1.} There exists $i \in I$ such that the vertex $\bar v_i$ of the cone $C(Y_i)/H_i$ belongs to $\bar B$.
	Then there is a homotopy $H : \bar B \times [0,1] \rightarrow \bar B$ which contracts $\bar B$ to $\{\bar v_i\}$.
	\subparagraph{Case 2.}
	The ball $\bar B$ does not contain a vertex $\bar v_i$.
	Thanks to the Cartan-Hadamard theorem, there exists a chart $(V,x)$ such that the developing map $(V, x) \rightarrow (\bar X, \bar x)$ induces an isometry from $B \left(x, \frac {r_0}{100}\right)$ onto its image.
	In particular the ball $\bar B$ lifts in one of the charts.
	Assume that this chart is $U_i = C(Y_i)/H_i$.
	Since  $\bar v_i$ does not belong to $\bar B$, the ball $\bar B$ lifts in fact in the cone $C(Y_i)$.
	By Proposition \ref{contracting ball contained in a cone}, $\bar B$ is contractible in $B(\bar x ,r +\bar \delta)$.
	In the other hand, if the chart $V$ is $U$, we apply  Proposition \ref {contracting ball in the cone-off}.
	Thus $\bar B$ is contractible in $B(\bar x, r+ 3 \bar \delta)$.

\paragraph{}Consequently, every ball $\bar B(\bar x,r)$ of $\bar X$,  with $r \leq 8(n +1)\bar \delta$, is contractible in $B(\bar x,r+4 \delta)$.
By Proposition \ref{aspherical hyperbolic complex}, $\bar X$ is contractible.
\end{proof}

%% file: 5-examples.tex
\section{Examples of aspherical complexes}
\label{part examples}

\paragraph{}In this section we explain how to construct examples of rotation families satisfying the very small cancellation assumptions.
Let $X$ be a proper, geodesic, $\delta$-hyperbolic space and $Y$ a closed convex subset of $X$. Let $G$ be a group acting properly, co-compactly, by isometries on $X$, such that $\stab Y$ acts co-compactly on $Y$.

\paragraph{} We are interested in the rotation family $\left(g Y, gH g^{-1} \right)_{g \in G/\stab Y}$, where $H$ is a subgroup of $\stab Y$.
\paragraph{} In concrete situations, $\left\{ \diam \left( gY^{+ 20\delta} \cap Y^{+20 \delta} \right), g \in G\setminus \stab Y \right\}$ may not be bounded. 
Nevertheless, in many situation, this assumption can be achieved by replacing $G$ with a finite index subgroup of $G$. 
This uses, as explained in the next lemma, the profinite topology of groups.

\begin{lemmsec}
\label{subgroup and overlap}
	Assume that the subgroup $\stab Y$ is closed in $G$ for the profinite topology.
	Let $\Delta$ be a non-negative number.
	There exists a finite index subgroup $G'$ of $G$ containing $\stab Y$ such that for all $g$ belongs to $G'$, $g \in \stab Y$ if and only if $\diam \left(g Y^{+20 \delta} \cap Y^{+20 \delta}\right) > \Delta$.
\end{lemmsec}

\begin{proof}
Let $K$ be a compact fundamental domain for the action of $\stab Y$ on $Y$. 
Since the action of $G$ is proper, the following set is finite.
\begin{displaymath}
	E= \left \{ g \in G / \diam \left ( g K ^{+ \Delta +30\delta } \cap K ^{+ \Delta +30 \delta} \right) > \Delta \right\}
\end{displaymath}
We assumed that $\stab Y$ was closed in $G$ for the profinite topology. In other words
\begin{displaymath}
	\stab Y = \bigcap_{\substack{N \triangleleft G \\ \left[ G:N\right] < + \infty}} \stab Y.N
\end{displaymath}
Hence there exists a finite index normal subgroup $N$ of $G$ such that $E \cap \stab Y . N  \subset \stab Y$. 
We denote by $G'$ the set $\stab Y . N$. 
It is a finite index subgroup of $G$ containing $\stab Y$. 
Consider now $g \in G'$, such that $\diam\left( gY^{+20 \delta} \cap Y^{+20 \delta}\right) > \Delta$. 
Since $K$ is a fundamental domain of $Y$,
there exist four points of $K^{+30\delta +\Delta}$, $x,x',y$ and $y'$ and two elements of $\stab Y$, $h$ and $h'$ with the following properties : 
$hx = gh'x'$, $hy = gh'y'$ and $\distX x y = \distX {x'} {y'} > \Delta$.
Thus $h^{-1}gh'$ belongs to $E\cap G'$. But $h$ and $h'$ both belong to $\stab Y$. Hence $g \in \stab Y$.
\end{proof}

In this context, the following result of N. Bergeron will be useful.
 
\begin{propsec}[{\cite[Lemme principal]{Ber00}}]
	\label{lemma bergeron}
	Let $\Lambda$ be an algebraic subgroup of $GL_n(\R)$ and $G$ a finitely generated subgroup of $GL_n(\R)$. Then $\Lambda \cap G$ is closed in $G$ for the profinite topology.
\end{propsec}

The second lemma explains how to find a subgroup $H$ of $\stab Y$ with an injectivity radius as large as desired.

\begin{lemmsec}
\label{subgroup and injectivity}
	Assume that the group $G$ is residually finite. 
	Then for all $\rho > 0$ there exists a finite index normal subgroup $H$ of $\stab Y$, such that $\rinj H X \geq \rho$.
\end{lemmsec}

\begin{proof}
	Let $K$ be a compact fundamental domain for the action of $G$ on $X$.
	Since the action of $G$ is proper, the following set is finite.
	\begin{displaymath}
		E = \left\{ g \in G / g K ^{+\rho} \cap K ^{+\rho} \neq \emptyset\right\}
	\end{displaymath}
	Moreover, $G$ is residually finite.
	Hence there exists a finite index normal subgroup $N$ of $G$ such that $E \cap N = \{1\}$.
	Consider $g \in N\setminus \{1\}$ and $x \in X$. 
	By definition there exists $h \in G$ such that $hx \in K$.
	But $h g h^{-1}$ belongs to $N \setminus \{1\}$, thus $hgh^{-1} K ^{+\rho} \cap K^{+\rho} = \emptyset$.
	It follows that
	\begin{math}
		\distX {x}{gx} = \distX{hx}{(hgh^{-1}) hx} \geq \rho
	\end{math}.
	Hence $[g] \geq \rho$.
	Consequently, we take for $H$ the group $N \cap \stab Y$.
\end{proof}

\begin{theosec}
	Let $\HP n$ denote the real (\resp complex, quaternionic) hyperbolic space, and $\delta$ its hyperbolicity constant. 
	We consider $\Lambda_k = SO(k,1)$ (\resp $SU(k,1)$, $Sp(k,1)$) as the stabilizer of $\HP k$ in $\HP n$.
	Let $G$ be a uniform lattice of $\Lambda_n=SO(n,1)$ (\resp $SU(n,1)$, $Sp(n,1)$). 
	We assume that $G \cap \Lambda_k$ is a uniform lattice of $\Lambda_k$. 
	\begin{enumerate}
		\item There exists a finite index subgroup $G'$ of $G$ such that the following set is bounded.
		\begin{displaymath}
			\left\{ \diam \right(g \HP k ^{+ 20 \delta} \cap \HP k^{+ 20 \delta}\left), g \in G'\setminus \Lambda_k\right\}
		\end{displaymath}
		\item Let $\bar Q$ be the space obtained by gluing a cone of base $\HP k / G' \cap \Lambda_k$ over $\HP n /G'$. 
		There is a finite index subgroup $H$ of $G' \cap \Lambda_k$ and a contractible hyperbolic space $\bar X$ such that $\bar G'  = G' / \ll H \gg$ acts properly co-compactly on $\bar X$, and $\bar Q = \bar X / \bar G'$.
	\end{enumerate}
\end{theosec}

\begin{proof}  Applying Proposition \ref{lemma bergeron}, $G\cap \Lambda_k$ is closed in $G$ for the profinite topology. 
The first point follows from Lemma \ref{subgroup and overlap}.

\paragraph{}
We denote by $\Delta $ the upper bound of 
\begin{displaymath}
	\left\{ \diam \left( g\HP k^{+ 20\delta} \cap \HP k^{+20 \delta} \right) , g \in G'\setminus \Lambda_k \right\}
\end{displaymath}
It is known that a finitely generated subgroup of $\Lambda_n$ is residually finite.
Using Lemma \ref{subgroup and injectivity} there exists a finite index normal subgroup $H$ of $G' \cap \Lambda_k$, whose injectivity radius is $\rho$, and such that $\delta \leq \delta_0 \rho$ and $\Delta \leq \Delta_0 \rho$, 
where $\delta_0$ and $\Delta_0$ are the constant given by the very small cancellation theorem (see Theorem \ref{very small cancellation}).
It follows that the rotation family $\left( g \HP k , g H g^{-1}\right)_{g\in G'/G'\cap \Lambda_k}$ satisfies the hypotheses of the very small cancellation theorem.
Thus there exists a hyperbolic space $\bar X$ such that $\bar G' = G'/ \ll H \gg$ acts properly by isometries on $\bar X$ and  $\bar Q  = \bar X/ \bar G'$.
Since $G'$ (\resp $H$) is a uniform lattice of $\Lambda_n$ (\resp $\Lambda_k$) the action of $G'$ (\resp $H$) on $\HP n$ (\resp $\HP k$) is co-compact. It follows that the action of $\bar G'$ on $\bar X$ is co-compact.
Moreover, every ball in $\HP n$ or $\HP k$ is contractible. 
Thus by applying Theorem \ref{asphericity universal cover}, $\bar X$ is contractible.
\end{proof}

The next result is proved in the same way, it only uses another kind of convex subsets of $\HP n$.

\begin{theosec}
	Let $\HC n$ denotes the complex hyperbolic space, and $\delta$ its hyperbolicity constant. 
	We consider $SO(n,1)$ as the stabilizer of the $n$-dimensional real hyperbolic space $\HR n$ in $\HC n$.
	Let $G$ be a uniform lattice of $SU(n,1)$ such that $G\cap SO(n,1)$ is also a uniform lattice of $SO(n,1)$.
	
	\begin{enumerate}
		\item There exists a finite index subgroup $G'$ of $G$ such that the following set is bounded.
		\begin{displaymath}
			\left\{ \diam \right(g \HR n ^{+ 20 \delta} \cap \HR n^{+ 20 \delta}\left), g \in G'\setminus SO(n,1)\right\}
		\end{displaymath}
		\item We denote by $\bar Q$ the space obtained by attaching a cone of base $\HR n / G'\cap SO(n,1)$ over $\HP n /G'$.
		There is a finite index subgroup $H$ of $G'\cap SO(n,1)$ and a contractible hyperbolic space $\bar X$ such that $\bar G'  = {G' / \ll H \gg}$ acts properly co-compactly on $\bar X$, and $\bar Q = \bar X / \bar G'$.
	\end{enumerate}
\end{theosec}
